\documentclass[11pt]{amsart}
\usepackage{adjustbox}
\usepackage{amssymb}
\usepackage{amsmath}
\usepackage{fancyhdr}
\usepackage[british]{babel}
\usepackage{geometry,mathtools}
\usepackage{enumitem}
\usepackage{algpseudocode}
\usepackage{dsfont}
\usepackage{centernot}
\usepackage{xstring}
\usepackage{colortbl}
\usepackage[section]{algorithm}
\usepackage{graphicx}
\usepackage[font={footnotesize}]{caption}
\usepackage[usenames,dvipsnames,table]{xcolor}
\usepackage{tikz}
\usepackage[h]{esvect}
\usepackage[
		bookmarksopen=true,
		bookmarksopenlevel=1,
		colorlinks=true,
		linkcolor=darkblue,
        linktoc=page,
		citecolor=darkblue,
]{hyperref}
\usepackage{bbm}
\usepackage{mathrsfs}
\usepackage{hyphenat}
\usepackage{soul}
\usepackage{cancel}
\usetikzlibrary{arrows.meta, decorations.pathmorphing, decorations, fit}

\allowdisplaybreaks

\title[Counting oriented trees in digraphs with large minimum semidegree]{Counting oriented trees in digraphs with large minimum semidegree}
\author[F.~Joos]{Felix Joos}
\address[F.~Joos]{Institut f\"ur Informatik, Universit\"at Heidelberg, Heidelberg,
Germany
}
\email{joos@informatik.uni-heidelberg.de}

\author[J.~Schrodt]{Jonathan Schrodt}
\address[J.~Schrodt]{Institut f\"ur Informatik, Universit\"at Heidelberg, Heidelberg,
	Germany
}
\email{schrodt@informatik.uni-heidelberg.de}

\date{\today}

\thanks{The research leading to these results was partially supported by the Deutsche Forschungsgemeinschaft (DFG, German Research Foundation) -- 428212407.
}

\newtheorem{theorem}[algorithm]{Theorem}

\newtheorem{lemma}[algorithm]{Lemma}
\newtheorem{cor}[algorithm]{Corollary}

\theoremstyle{definition}

\newtheorem{defin}[algorithm]{Definition}

\newtheoremstyle{claimstyle}{5pt}{5pt}{\em}{5pt}{\em}{:}{5pt}{}
\theoremstyle{claimstyle}

\newtheoremstyle{stepstyle}{10pt}{5pt}{\em}{0pt}{\em}{:}{5pt}{}
\theoremstyle{stepstyle}

\numberwithin{equation}{section}

\definecolor{darkblue}{rgb}{0,0,0.5}

\newdimen\margin
\def\textno#1&#2\par{
   \margin=\hsize
   \advance\margin by -4\parindent
          \setbox1=\hbox{\sl#1}
   \ifdim\wd1 < \margin
      $$\box1\eqno#2$$
   \else
      \bigbreak
      \hbox to \hsize{\indent$\vcenter{\advance\hsize by -3\parindent
      \it\noindent#1}\hfil#2$}
      \bigbreak
   \fi}

\DeclarePairedDelimiter{\abs}{\lvert}{\rvert}

\begin{document}

\newcommand{\new}[1]{\textcolor{red}{#1}}
\newcommand{\old}[1]{\textcolor{red}{\st{#1}}}
\newcommand{\oldblock}[1]{\begin{tikzpicture}\node[inner sep=0pt, outer sep=0pt] at (0,0) (bb) {\begin{minipage}{\textwidth}\color{red}#1\end{minipage}};\draw[-, red] (bb.north west) -- (bb.south east);\end{tikzpicture}}
\def\COMMENT#1{}
\def\TASK#1{}

\newcommand{\todo}[1]{\begin{center}\color{red}\textbf{to do:} #1 \end{center}}

\def\op{\operatorname}
\def\eps{\varepsilon}

\newcommand{\pr}{\mathbb{P}}
\newcommand{\ex}{\mathbb{E}}

\newcommand{\bN}{\mathbb{N}}
\newcommand{\bZ}{\mathbb{Z}}
\newcommand{\bQ}{\mathbb{Q}}
\newcommand{\bR}{\mathbb{R}}
\newcommand{\bC}{\mathbb{C}}

\newcommand{\bx}{\mathbf{x}}

\newcommand{\cE}{\mathcal{E}}
\newcommand{\cG}{\mathcal{G}}
\newcommand{\cS}{\mathcal{S}}
\newcommand{\cT}{\mathcal{T}}
\newcommand{\cQ}{\mathcal{Q}}
\newcommand{\cP}{\mathcal{P}}
\newcommand{\cA}{\mathcal{A}}

\newcommand{\1}{\mathbbm 1}
\newcommand{\supp}{{\rm supp}}
\newcommand{\diff}{\mathop{}\!\mathrm{d}}

\newcommand\restrict[1]{\raisebox{-.5ex}{$|$}_{#1}}

\newcommand{\Set}[1]{\{#1\}}
\newcommand{\set}[2]{\{#1\,:\;#2\}}

\newcommand{\norm}[1]{\|#1\|}

\newcommand{\AU}[1]{\op{Aut}(#1)}

\begin{abstract} 
Let~$T$ be an oriented tree on~$n$ vertices with maximum degree at most~$e^{o(\sqrt{\log n})}$.
If~$G$ is a digraph on~$n$ vertices with minimum semidegree~$\delta^0(G)\geq(\frac12+o(1))n$, 
then~$G$ contains~$T$ as a spanning tree, as recently shown by Kathapurkar and Montgomery (in fact, they only require maximum degree $o(n/\log n)$).
This generalizes the corresponding result by Komlós, Sárközy and Szemerédi for graphs.
We investigate the natural question how many copies of~$T$ the digraph~$G$ contains.
Our main result states that every such~$G$ contains at least~$|\op{Aut(T)}|^{-1}(\frac12-o(1))^nn!$ copies of~$T$, which is optimal.
This implies the analogous result in the undirected case.
\end{abstract}

\maketitle
\section{Introduction}

One of the cornerstones of extremal combinatorics is Dirac's theorem stating that every graph~$G$ on $n\geq 3$ vertices with minimum degree $\delta(G)\geq \frac{n}{2}$ contains a Hamilton cycle.
On the one hand, the bound on the minimum degree is sharp as there are graphs with minimum degree~$\frac{n-1}{2}$ that do not contain any Hamilton cycle.
On the other hand,
there are numerous results testifying that any graph~$G$ with $\delta(G)\geq\frac{n}{2}$ contains Hamilton cycles in an extremely strong and robust sense.
This includes random sparsifications of $G$~\cite{KLS:14}, colour or conflict restrictions on the edges of $G$~\cite{KLS:17},
and the fact, which is proved by Cuckler and Kahn~\cite{CK:09}, that such graphs contain at least $(\frac12-o(1))^nn!$ Hamilton cycles.
This lower bound is easily seen to be optimal by considering the binomial random graph.
To be more precise, for every labelled graph $H$ on $n$ vertices, $K_n$ contains $n!$ copies of $H$ and $|\AU{H}|^{-1}n!$ unlabelled copies of $H$, where $\AU{H}$ refers to the automorphism group of $H$.
Hence $G(n,p)$ contains $p^{e(H)}|\AU{H}|^{-1}n!$ unlabelled copies of $H$ in expectation and by Markov's inequality, 
there is a graph~$G$ with $\delta(G)\geq (1+o(1))pn$ and at most $(1+o(1))^np^{e(H)}|\AU{H}|^{-1}n!$ subgraphs isomorphic to~$H$.

For all of the above results, it does not really matter whether we ask for a Hamilton cycle or a Hamilton path.
As a Hamilton path is potentially the simplest spanning tree,
one may ask whether the above results are also true for other spanning trees.
In this paper we solve this problem for the counting question even in the case where the trees have vertices of substantially large degree.

\begin{theorem}\label{thm:main_simple}
Suppose~$G$ is a graph on~$n$ vertices with $\delta(G)\geq (\frac12+o(1))n$ and~$T$ is a tree on~$n$ vertices with $\Delta(T)\leq e^{o(\sqrt{\log n})}$.
Then~$G$ contains at least $|\AU{T}|^{-1}(\frac12-o(1))^nn!$ copies of~$T$.
\end{theorem}

Observe that the error term in the minimum degree bound on~$G$ cannot be completely omitted in contrast to Dirac's theorem.
Koml\'os, S\'ark\"ozy and Szemer\'edi proved that $\delta(G)\geq (\frac12+o(1))n$ suffices to guarantee any bounded degree tree on~$n$ vertices as a subgraph, confirming a conjecture of Bollob\'as, and they
later proved that under the same condition, $G$ even contains any tree~$T$ on~$n$ vertices with $\Delta(T)\leq o(\frac{n}{\log n})$~\cite{KSS:95, KSS:01}.
Hence our result also strengthens these results for trees $T$ with $\Delta(T)\leq e^{o(\sqrt{\log n})}$.

We note further that there are trees $T$ where $|\op{Aut}(T)|$ grows exponentially in $n$, for example, the balanced binary tree.
However, if $T$ is a path or a cycle, then $|\op{Aut}(T)|\leq 2n$ and 
hence this quantity is smaller than the error term and can be neglected. 

Cuckler and Kahn proved a significantly stronger result than Theorem~\ref{thm:main_simple}.
To state their result, we say 
a function $\bx\colon E(G)\to\bR_{\geq0}$ is an \emph{edge weighting} of $G$ and
write~$\bx_e$ for~$\bx(e)$.
We call an edge weighting a \emph{perfect fractional matching} if $\sum_{w\in N(v)} \textbf{x}_{vw} = 1$ for every vertex $v\in~V(G)$.
Moreover, let $h(G)=\max_\bx \sum_{e\in E(G)}-\bx_e\log \bx_e$ where 
the maximum is taken over all perfect fractional matchings of~$G$.
Here and throughout the paper, we write $\log = \log_2$.

\begin{theorem}[Cuckler and Kahn~{\cite[Theorem~1.2]{CK:09}}]\label{thm:CK}
Suppose~$G$ is a graph on~$n$ vertices with $\delta(G)\geq \frac{n}{2}$.
Then~$G$ contains at least $2^{2h(G)-n\log e-o(n)}$ Hamilton cycles.
\end{theorem}
In fact, Cuckler and Kahn showed in another paper that the lower bound in Theorem~\ref{thm:CK} is also an upper bound (up to the error term)~\cite{CK:09 upper bound}.

Our main theorem implies the generalisation of Theorem~\ref{thm:CK} to spanning trees.

\begin{theorem}\label{thm:main_undirected}
Suppose~$G$ is a graph on~$n$ vertices with $\delta(G)\geq (\frac12+o(1))n$ and~$T$ is a tree on~$n$ vertices with $\Delta(T)\leq e^{o(\sqrt{\log n})}$.
Then~$G$ contains at least $|\AU{T}|^{-1}2^{2h(G)-n\log e-o(n)}$ copies of~$T$.
\end{theorem}

Note that $h(G)\geq \frac{n}{2}\log \delta(G)$ for any graph~$G$ on~$n$ vertices with $\delta(G)\geq \frac{n}{2}$~\cite[Theorem~1.3]{CK:09}
and hence Theorem~\ref{thm:main_undirected} implies Theorem~\ref{thm:main_simple}.

Clearly, the lower bound given in Theorems~\ref{thm:main_simple} and~\ref{thm:main_undirected} cannot be improved
as we may choose~$T$ be to be Hamilton path.
However, note that Theorems~\ref{thm:main_simple} and~\ref{thm:main_undirected} are tight for \emph{every} such tree.

\medskip

Recently, Kathapurkar and Montgomery transferred the above-mentioned theorem due to Koml\'os, S\'ark\"ozy and Szemer\'edi to digraphs thereby implying their result and
extending earlier works of Mycroft and Naia~\cite{MN:20}.
To state their theorem,
let~$\delta^0(G)$ be the \emph{minimum semidegree} of a digraph~$G$
which is the maximum~$k$ such that every vertex has at least~$k$ in-neighbours and~$k$ out-neighbours
and let~$\Delta(G)$ be the maximum degree of the underlying undirected graph (counting parallel edges).

\begin{theorem}[Kathapurkar and Montgomery {\cite[Theorem~1.1]{KM:22}}]
Suppose $G$ is a digraph on~$n$ vertices with $\delta^0(G)\geq (\frac12+o(1))n$ 
and~$T$ is an oriented tree on~$n$ vertices with $\Delta(T)\leq o(\frac{n}{\log n})$.
Then~$G$ contains a copy of~$T$.
\end{theorem}

We also generalise their theorem up to a more restrictive bound on the maximum degree.
For a digraph~$G$ with vertex set $\{v_1,\ldots,v_n\}$,
let~$B_G$ the bipartite graph corresponding to~$G$,
that is, the vertex set of~$B_G$ is 
$\{v_1^+,v_1^-,\ldots,v_n^+,v_n^-\}$ and the edge set is $\{v_i^+v_j^- \mid v_iv_j\in E(G)\}$.
We define $h(G):=h(B_G)$ to be the entropy of~$G$.
The following theorem is our main result.

\begin{theorem}\label{thm:main_directed}
Suppose~$G$ is a digraph on~$n$ vertices with $\delta^0(G)\geq (\frac12+o(1))n$ and~$T$ is an oriented tree on~$n$ vertices with $\Delta(T)\leq e^{o(\sqrt{\log n})}$.
Then~$G$ contains at least $|\AU{T}|^{-1}2^{h(G)-n\log e-o(n)}$ copies of~$T$.
\end{theorem}

Again, the bound on the number of copies of~$T$ is optimal up to the error term.
Observe that Theorem~\ref{thm:main_directed} implies Theorem~\ref{thm:main_undirected}
by replacing any edge in the undirected graph by two antidirected parallel edges and orienting~$T$ arbitrarily.

Note that $h(G)\geq n\log \delta^0(G)$ whenever $\delta^0(G)\geq \frac{n}{2}$~\cite[Theorem~3.1]{CK:09}. This yields the more explicit lower bound of $|\op{Aut}(T)|^{-1}(\frac12-o(1))^nn!$ copies of~$T$ in~$G$.

Our paper is organised as follows.
In Section~\ref{sec:sketch},
we give an outline of our proof method.
In Section~\ref{sec:Preliminaries},
we introduce several basic concepts.
In Section~\ref{sec:Normal perfect fractional matchings with large entropy},
we prove that there are perfect fractional matchings that almost yield the maximal entropy and are particularly well distributed.
In Section~\ref{sec:Conditions for a random tree being well-behaved},
random trees are considered that yield the desired partial embedding of our tree
and in Section~\ref{sec:Counting trees},
we prove our main theorem.

\section{Proof Sketch}\label{sec:sketch}

We explain our idea for the proof of Theorem~\ref{thm:main_directed} by explaining the proof idea of Theorem~\ref{thm:main_undirected} and outlining the additional complications when transitioning from graphs to digraphs later.
So suppose we are in the setting of Theorem~\ref{thm:main_undirected}, 
that is, we are given a graph~$G$ on~$n$ vertices with $\delta(G)\geq (\frac12+o(1))n$ and a tree~$T$ on~$n$ vertices with an appropriate bound on~$\Delta(T)$.
We first partition~$T$ into a collection of subtrees $T',T_1,\ldots,T_k$ such that $T'\cup T_1 \cup \ldots \cup T_i$ is a tree for all~$i$
and~$T'$ is only adjacent to~$T_1$.
Now we proceed in an inductive manner and
show that there are suitably many embeddings of~$T_{i+1}$ given we already fixed any embedding of $T_1 \cup \ldots \cup T_i$ in~$G$.
At the end we embed~$T'$ and complete the embedding.
This is possible by, before embedding $T_1,\ldots,T_k$, taking aside a small vertex set with particular properties which ensures that we can always embed~$T'$ (for this last step we utilise a result due to Kathapurkar and Montgomery).

In order to show one step of our induction,
we embed~$T_{i+1}$ randomly.
To be more precise, we interpret the weights of a perfect fractional matching of~$G$ as transition probabilities of a random walk-like structure.
If~$T_{i+1}$ is a path, we simply consider a random walk in $G-(T_1\cup \ldots \cup T_i)$,
otherwise, the random walk branches suitably according to the structure of~$T_{i+1}$.
The entropy of this random embedding is linked to the entropy of~$G$ and thus we obtain a lower bound on the number of realisations of the random embedding of~$T_{i+1}$ with the little caveat that we want to condition on the event that the realisations have random-like properties (this is what we call \emph{well-behaved}).
As we show in Lemma~\ref{lemma: entropy}, the entropy decreases only slightly by conditioning on a likely event.

The sizes of the trees~$T_{i+1}$ have to be chosen large enough that one can apply concentration inequalities in order to verify that most realisations of random embeddings are in fact well-behaved and 
the sizes of the~$T_{i+1}$ have to be small enough in order to be able to exploit that at every step of the random embedding we can still use essentially all vertices in $G-(T_1\cup \ldots \cup T_i)$.

On a high level, Cuckler and Kahn already followed this path and 
hence our proof shares ideas with their proof.
However, focusing only on cycles simplifies the arguments at several places considerably and new ideas are needed.

In order to move from graphs to digraphs,
we mainly have to adapt the definition of a random tree, as the direction of the edge between two in~$T$ adjacent vertices determines whether we use an in-neighbour or an out-neighbour of the image of the previously embedded vertex as the image of its neighbour.

\section{Preliminaries}\label{sec:Preliminaries}
\subsection{Notation}

For an integer~$n\geq 1$, let $[n]:=\{1,\ldots,n\}$ and $[n]_0:=[n]\cup \{0\}$.
For real numbers~$\alpha,\beta,\delta$, we write $\alpha=(1\pm\delta)\beta$ if $(1-\delta)\beta\leq\alpha\leq(1+\delta)\beta$.
If we write $\alpha\ll\beta$, where $\alpha,\beta\in(0,1]$, there is a non-decreasing function $\alpha_0\colon(0,1]\to(0,1]$ such that for any $\beta\in(0,1]$, the subsequent statement holds for $\alpha\in(0,\alpha_0(\beta)]$.
Hierarchies with more constants are defined in a similar way and should be read from right to left.
If~$\frac{1}{n}$ appears in a hierarchy, this implicitly means that~$n$ is a natural number.
We ignore rounding issues when they do not affect the argument.

We write~$\log:=\log_2$ and $\ln:=\log_e$.
For $x\geq0$, we often use the inequality $\log (1+x) \leq 2x$ that follows immediately from $1+x \leq e^x \leq 2^{2x}$.

\subsection{Entropy}
Let~$X$ be a discrete random variable with possible outcomes~$x_1,\ldots,x_N$.
Let~$p_i:=\pr[X=x_i]$.
The \emph{entropy} of~$X$ is defined as
\begin{equation*}
	H(X) := -\sum_{i=1}^N p_i\log p_i = \sum_{i=1}^N p_i\log\frac{1}{p_i},
\end{equation*}
where we set $0\log0:=0$. If~$\cE$ is an event with nonzero probability and $p_i':=\pr[X=x_i\mid \cE]$, we define by $H(X\mid\cE) := -\sum_{i=1}^N p_i' \log p_i'$
the \emph{conditional entropy of~$X$ given~$\cE$}.

We define $\op{supp} X := \{i\in[N]\mid p_i\neq0\}$ and $\op{supp}(X\mid\cE) := \{i\in[N]\mid p_i'\neq0\}$.
If~$Y$ is a discrete random variable with possible outcomes $y_1,\ldots,y_M$, we define by $H(X\mid Y) := \sum_{i\in\op{supp} Y}\pr[Y=y_i]H(X\mid Y=y_i)$ the \emph{conditional entropy of~$X$ given~$Y$}.

By Jensen's inequality the entropy of a random variable is maximised by the uniform distribution, that is, $H(X) \leq \log |\op{supp} X|$ and $H(X\mid\cE) \leq \log|\op{supp}(X\mid \cE)|$. In particular, $H(X)\leq \log N$.

The following lemma shows that~$H(X\mid\cE)$ cannot be much smaller than~$H(X)$ if~$\cE$ is an event that is very likely and if~$p_i$ is not too small for all $i\in [N]$.
\begin{lemma}\label{lemma: entropy}
	Let~$X$ be a random variable with values in $\{x_1,\ldots,x_N\}$.
	Let $A\geq16$ be such that $\pr[X=x_i]\geq\frac{1}{A}$ for all $i\in[N]$.
    Let $J\subseteq\{x_1,\ldots,x_N\}$ and let~$\cE$ be the event that $X\in J$.
	Suppose that $\pr[\cE] \geq 1-a$ for some $0\leq a\leq \frac12$. Then
	\begin{equation*}
		H(X) - H(X\mid\cE) \leq 2a\log A.
	\end{equation*}
\end{lemma}
\begin{proof}
	We define $p_i:=\pr[X=x_i]$ and $p_i':=\pr[X=x_i\mid \cE]$ as well as~$I:=\op{supp}(X\mid\cE)=\{i\in[N]\mid x_i\in J\}$. Note that~$p_i'=0$ if~$i\not\in I$ and~$p_i\leq p_i' = \frac{p_i}{\pr[\cE]} \leq \frac{p_i}{1-a}$ if~$i\in I$. Thus, $p_i'\log p_i'\leq p_i\log p_i'\leq p_i\log\frac{p_i}{1-a}$ for $i\in I$. In particular,
    \begin{align*}
        -\sum_{i\in I} p_i\log p_i + \sum_{i\in I} p_i'\log p_i'
        \leq -\sum_{i\in I} p_i\log p_i + \sum_{i\in I} p_i\log\frac{p_i}{1-a}
        = \log\frac{1}{1-a}\cdot\sum_{i\in I}p_i \leq \log\frac{1}{1-a}.
    \end{align*}
    Combining this with $-\sum_{i\in [N]\setminus I}p_i\log p_i \leq \log A\cdot\sum_{i\in[N]\setminus I} p_i \leq a\log A$, we obtain
	\begin{align*}
		H(X) - H(X\mid\cE)
		&= -\sum_{i\in[N]} p_i\log p_i + \sum_{i\in [N]} p_i'\log p_i' \\
		&= -\sum_{i\in I} p_i\log p_i + \sum_{i\in I} p_i'\log p_i' - \sum_{i\in[N]\setminus I} p_i\log p_i \\
		&\leq \log \frac{1}{1-a} + a\log A.
	\end{align*}
	Using $\log\frac{1}{1-a}\leq \log(1+2a) \leq 4a \leq a\log A$ yields the desired result.
\end{proof}

\subsection{Graphs}
Let~$G$ be a graph.
By~$V(G)$ we denote the vertex set and by~$E(G)$ the edge set of~$G$, respectively.
The \emph{order} of~$G$ is the number of its vertices (sometimes we simply write~$\abs{G}$ instead of~$\abs{V(G)}$).
The \emph{size} of~$G$ is the number of its edges.
For an edge~$\{v,w\}\in E(G)$ we simply write~$vw$ or~$wv$.
The \emph{neighbourhood} of a vertex $v\in V(G)$ is defined as $N(v):=N_G(v):=\{w\in V(G)\mid vw\in E(G)\}$.
The \emph{degree} of a vertex $v\in V(G)$ is $\deg(v):=|N(v)|$.
If $S\subseteq V(G)$ and $v\in V(G)$, we denote $N_{G,S}(v):=N_G(v)\cap S$.
For a subgraph~$H$ of~$G$, we define $N_{G,H}(v):=N_{G,V(H)}(v)$.
We write~$\delta(G)$ for the minimum degree and~$\Delta(G)$ for the maximum degree of a graph~$G$, respectively.
An~\emph{\mbox{$(n,\eps)$-graph}} is a graph~$G$ on~$n$ vertices with~$\delta(G)\geq(\frac12+\eps)n$.

If~$G$ is a graph and $S\subseteq V(G)$, we write~$G[S]$ for the subgraph induced by~$S$ in~$G$, that is, $G[S]$ is the graph with vertex set $S$ and edge set $\{vw\in E(G)\mid v,w\in S\}$. For a vertex set $S\subseteq V(G)$ and induced subgraphs $H,H'\subseteq G$, we define $G-S:=G[V(G)\setminus S]$, $G-H:=G-V(H)$, as well as $H\cup S:=G[V(H)\cup S]$ and $H\cup H':=G[V(H)\cup V(H')]$.

Let~$G,H$ be graphs and $\varphi\colon V(G)\to V(H)$.
We call~$\varphi$ an \emph{isomorphism}
if~$\varphi$ is bijective and $vw\in E(G)$ if and only if $\varphi(v)\varphi(w)\in E(H)$.
We say that~$G$ and~$H$ are \emph{isomorphic}
if there exists an isomorphism between~$G$ and~$H$.
If~$G,H$ are graphs and~$H'$ is a subgraph of~$G$ which is isomorphic to~$H$, we call~$H'$ a \emph{copy} of~$H$ in~$G$.
We call an isomorphism $\varphi\colon V(G)\to V(G)$ an \emph{automorphism} of~$G$.
The automorphisms of a graph~$G$ form a group~$\op{Aut}(G)$, the \emph{automorphism group} of~$G$.

\subsection{Digraphs}
Let~$G$ be a digraph. For an edge $(v,w)\in E(G)$, we simply write~$vw$.
We denote the \emph{out-neighbourhood} of a vertex $v\in V(G)$ as $N^+(v):=N^+_G(v):=\{w\in V(G)\mid vw\in E(G)\}$ and the \emph{out-degree} of~$v$ by $\deg^+(v):=|N^+(v)|$.
Analogously, we define the \emph{in-neighbourhood} as $N^-(v):=N^-_G(v):=\{w\in V(G)\mid wv\in E(G)\}$ and the \emph{in-degree} of~$v$ by $\deg^-(v):=|N^-(v)|$.
By $\delta^+(G)$ we denote the \emph{minimum out-degree} and by $\delta^-(G)$ the \emph{minimum in-degree} of~$G$, respectively.
By $\delta^0(G):=\min_{v\in V(G)}\{\deg^+(v),\deg^-(v)\}$ we denote the \emph{minimum semidegree} of~$G$.
An \mbox{$(n,\eps)$-digraph} is a digraph~$G$ on~$n$ vertices with $\delta^0(G)\geq(\frac12+\eps)n$.
The other notions defined for graphs are transferred easily to digraphs.

Note that any digraph~$G$ with vertex set $V(G):=\{v_1,\ldots,v_n\}$ corresponds to a balanced bipartite graph~$B_G$ with vertex set $V(B_G):=\{v_1^+,v_1^-,\ldots,v_n^+,v_n^-\}$ and edge set $E(B_G) := \{v_i^+v_j^- \mid v_iv_j\in E(G)\}$.

\subsection{Edge weightings and entropy of edge weightings}
Let~$G$ be a (di)graph on~$n$ vertices.
A function $\textbf{x}\colon E(G)\to\bR_{\geq0}$ is called an \emph{edge weighting}.
We write~$\textbf{x}_e$ for~$\textbf{x}(e)$.
An edge weighting is \mbox{\emph{$b$-normal}} if $\frac{1}{bn} \leq \textbf{x}_e \leq \frac{b}{n}$ for each $e\in E(G)$.
We call an edge weighting a \emph{perfect fractional matching} if $\sum_{w\in N(v)} \textbf{x}_{vw} = 1$ for every vertex $v\in~V(G)$ in case~$G$ is a graph and if $\sum_{w\in N^+(v)}\textbf{x}_{vw}=1=\sum_{w\in N^-(v)}\textbf{x}_{wv}$ for every vertex $v\in~V(G)$ in case~$G$ is a digraph (in this case, $\textbf{x}$, extended to~$B_G$ in the obvious way, is a perfect fractional matching of~$B_G$).
Note that any perfect fractional matching~\textbf{x} of a graph~$G$ on~$n$ vertices satisfies $\sum_{e\in E(G)}\textbf{x}_e = \frac{n}{2}$.

Let~$\textbf{x}$ be an edge weighting of a (di)graph~$G$.
If~$M\subseteq E(G)$, we define by
\begin{equation*}
	h_{\textbf{x}}(M) := \sum_{e\in M}\textbf{x}_e\log\frac{1}{\textbf{x}_e}
\end{equation*}
the \emph{entropy of~$M$}. We define the entropy of~$\textbf{x}$ by~$h(\textbf{x}) := h_{\textbf{x}}(E(G))$ and the entropy of~$G$ by~$h(G):=\max\{h(\textbf{x}) \mid \text{$\textbf{x}$ is a perfect fractional matching of~$G$}\}$.
Note that any perfect fractional matching~$\textbf{x}$ of a graph~$G$ on~$n$ vertices satisfies $h(\textbf{x})\leq \frac{n}{2}\log n$ by Jensen's inequality, and $h(\textbf{x})\leq n\log n$ in case~$G$ is a digraph. Note further that any $b$-normal perfect fractional matching~$\textbf{x}$ satisfies $h(\textbf{x}) \geq \frac{n}{2}\log\frac{n}{b}$ in case~$G$ is a graph and $h(\textbf{x})\geq n\log\frac{n}{b}$ in case~$G$ is a digraph.

For $v\in V(G)$ and a perfect fractional matching~$\textbf{x}$, we define by
\[
    h_{\textbf{x}}(v) := \sum_{w\in N(v)} \textbf{x}_{vw}\log\frac{1}{\textbf{x}_{vw}}
\]
the \emph{entropy of $v$} and by~$h_{\textbf{x}}^+(v)$ and~$h_{\textbf{x}}^-(v)$, respectively, the analogous notions for digraphs.

\subsection{Markov chains and random walks}
For a set~$S=\{s_1,\ldots,s_n\}$ and real matrices~$P^{(t)}=(p^{(t)}_{ij})_{i,j\in[n]}$, we say that a sequence of random variables~$Z=(Z_t)_{t\in\bN_0}$ is a~\emph{Markov chain} with \emph{state space}~$S$ and \emph{transition matrix}~$P^{(t)}$ at time~$t$ if
\begin{equation*}
	\pr[Z_{t+1}=s_{i_{t+1}}\mid Z_0=s_{i_0},\ldots,Z_t=s_{i_{t}}]=\pr[Z_{t+1}=s_{i_{t+1}}\mid Z_{t}=s_{i_{t}}]=p^{(t)}_{i_ti_{t+1}}
\end{equation*}
holds for all~$t\in\bN_0$ and~$i_0,\ldots,i_{t+1}\in [n]$.

Let~$G$ be a graph and~$\textbf{x}$ a perfect fractional matching of~$G$. A \emph{random walk on~$G$ induced by~$\textbf{x}$} is the Markov chain~$(Z_t)_{t\in\bN_0}$ with state space~$V(G)$ and transition matrix~$(p_{vw})_{v,w\in V(G)}$, where~$p_{vw}=\pr[Z_{i+1}=w\mid Z_i=v]$ for any~$i\in\bN_0$ is defined by~$p_{vw}:=\textbf{x}_{vw}$ if~$vw\in E(G)$ and $p_{vw}:=0$ if $vw\not\in E(G)$.

Let~$G$ be a digraph, let~$\textbf{x}$ be a perfect fractional matching of~$G$ and let~$L=(y_t)_{t\in\bN_0}$ be an oriented path (which is not a subgraph of~$G$).
A \emph{random walk on~$G$ induced by~$\textbf{x}$ following the pattern of~$L$} is the
Markov chain $(Z_t)_{t\in\bN_0}$ with state space~$V(G)$ and transition matrix $P^{(t)}=(p^{(t)}_{vw})_{v,w\in V(G)}$ at time~$t$, where $p^{(t)}_{vw}=\pr[Z_{t+1}=w\mid Z_t=v]$ is defined by
\[
    p_{vw}^{(t)} := \left\{\begin{matrix*}[l]
        \textbf{x}_{vw} & \text{if $y_ty_{t+1}\in E(L)$ and $vw\in E(G)$,}  \\ 
        \textbf{x}_{wv} & \text{if $y_{t+1}y_t\in E(L)$ and $wv\in E(G)$,} \\
        0 & \text{otherwise.} \end{matrix*}\right.
\]

\subsection{Azuma's inequality}

A sequence~$Z_0,\ldots,Z_k$ of random variables is a \emph{martingale} if~$Z_0$ is a fixed real number and~$\ex[Z_i\mid Z_0,\ldots,Z_{i-1}]=Z_{i-1}$ for all~$i\in[k]$.

We will need the following inequality by Azuma, see for example~\cite{azuma:67}.
\begin{lemma}\label{lemma: Azuma}
	Suppose~$Z_0,\ldots,Z_k$ is a martingale and $|Z_i-Z_{i-1}|\leq c$ for all~$i\in[k]$. Then for all~$t>0$, we have
	\begin{equation*}
		\pr[|Z_k-Z_0|\geq t] \leq 2\exp\left( -\frac{t^2}{2kc^2} \right).
	\end{equation*}
\end{lemma}

\subsection{Trees and tree decomposition}
We write $(T,r_0)$ for a rooted tree~$T$ with root~$r_0$.
Embedding a rooted tree~$(T,r_0)$ in the plane gives a natural breadth-first ordering $(r_0,\ldots,r_m)$ of the vertices of~$T$.
By $T=(r_0,\ldots,r_m)$ we denote a rooted tree with fixed breadth-first order.
The \emph{depth} of a vertex $v\in(T,r)$ is the length of the unique path from~$r$ to~$v$.
For a vertex~$v$ in a rooted tree~$(T,r)$, let~$T(v)$ be the subtree of~$T$ containing all vertices below~$v$ (including~$v$),
that is, all vertices~$u$ such that the path between~$u$ and~$r$ in~$T$ contains~$v$.
We need the following simple lemma proven in \cite{JKKO:19}.
\begin{lemma}[{\cite[Proposition~6.5]{JKKO:19}}] \label{lemma: tree partition}
    Suppose $n,\Delta\geq2$ and $n\geq t\geq1$. Then for any rooted tree~$(T,r)$ on~$n$ vertices with $\Delta(T)\leq\Delta$, there exists a collection $(T_1,t_1),\ldots,(T_k,t_k)$ of pairwise vertex-disjoint rooted subtrees such that
    \begin{enumerate}[label=\normalfont(\roman*)]
        \item $V(T) = \bigcup_{i=1}^k V(T_i)$,
        \item $T_i\subseteq T(t_i)$ for every $i\in[k]$,
        \item the depth of $t_1,\ldots,t_k$ is non-decreasing, and
        \item $t\leq|T_i|\leq2\Delta t$ for every $i\in[k]$.
    \end{enumerate}
\end{lemma}
Note that, in particular, $t_1=r$ and $T_1\cup\ldots\cup T_i$ is a tree for every $i\in[k]$.
The construction of the subtrees is done in reverse order by choosing~$y$ at maximal distance from~$r$ subject to $|T(y)|\geq t$. An easy modification of this argument yields the following claim.

\begin{lemma}\label{lemma: helping lemma tree decomposition}
    Suppose $n,\Delta\geq2$. Let~$(T,r)$ be a rooted tree on~$n$ vertices and $\Delta(T)\leq\Delta$. Let $n_0\in\bN$ such that $n \leq n_0 < n^4$. Then there exist pairwise vertex-disjoint rooted subtrees $(T_1,t_1),\ldots, (T_k,t_k)$ such that, denoting $n_i:=n_0-|T_1\cup\ldots\cup T_i|+1$,%
    \COMMENT{
    Let $(Q_1,q_1)\ldots,(Q_{s-1},q_{s-1})$ be already constructed subtrees of~$T$ by the process described below.
    In this case, $T':=T-(Q_1\cup\ldots\cup Q_{s-1})$ is a tree. Let $\ell_i := n_0 - n +|Q_1\cup\ldots\cup Q_i| + 1$ for $i\in[s-1]_0$.
    We want to construct~$Q_s$ such that $\ell_s^\frac14\leq|Q_s|\leq2\Delta\ell_s^\frac14$, where $\ell_s:=\ell_{s-1}+|Q_s|$.
    We choose~$q_s$ at maximal distance from~$r$ subject to $|T'(q_s)|\geq m$,
    where~$m$ is the smallest natural number such that $m^4-(m+\ell_{s-1})\geq0$ (note that this implies the depth of~$q_s$ is at most the depth of~$q_{s-1}$).
    Then $m\leq |T'(q_s)|\leq\Delta m$.
    Let $Q_s:=T'(q_s)$.
    As $m\geq2$, we have $m^4\leq(2(m-1))^4$, thus, $m^4-2^4m-2^4((m-1)^4-(m-1))\leq0$.
    Together with the inequality $\ell_{s-1}\geq(m-1)^4-(m-1)$, this yields $m^4-2^4(m+\ell_{s-1})\leq0$.
    Thus, $(m+\ell_{s-1})\leq m^4\leq 2^4(m+\ell_{s-1})$, that is, $(m+\ell_{s-1})^\frac14\leq m\leq 2(m+\ell_{s-1})^\frac14$.
    Then $\ell_s^\frac14=(|Q_s|+\ell_{s-1})^\frac14\leq|Q_s|$, as~$Q_s$ satisfies $|Q_s|\geq m$ and thus $|Q_s|^4-(|Q_s|+\ell_{s-1})\geq 0$.
    Further, $|Q_s|\leq \Delta m \leq 2\Delta(m+\ell_{s-1})^\frac14 \leq 2\Delta(|Q_s|+\ell_{s-1})^\frac14 = 2\Delta\ell_s^\frac14$, as desired.\\
    Let $r_{s-1} := n-|Q_1\cup\ldots\cup Q_{s-1}|$ denote the number of not yet used vertices in the tree decomposition.
    Assume we come to a point where the smallest natural number~$m$ with $m^4-(m+\ell_{s-1})\geq0$ satisfies $m > r_{s-1}$ (in this case it is not possible to construct a subtree~$Q_s$ with $|Q_s|\geq m$, as this would imply $|T| = n = r_{s-1} + |Q_1\cup\ldots\cup Q_{s-1}| < m + |Q_1\cup\ldots\cup Q_{s-1}| \leq |Q_1\cup\ldots\cup Q_s|\leq|T|$).
    In particular, we have $r_{s-1} < (r_{s-1}+\ell_{s-1})^{\frac14} = (n_0+1)^{\frac14}$.
    Note that $\ell_{s-1}^\frac14\leq |Q_{s-1}|\leq 2\Delta \ell_{s-1}^\frac14$.
    Thus, $n_0^\frac14 \leq |Q_{s-1}|+r_{s-1} \leq \frac52\Delta n_0^\frac14$
    (the first inequality follows from $(|Q_{s-1}|+r_{s-1})^4 \geq (\ell_{s-1}^\frac14+r_{s-1})^4 \geq \ell_{s-1}+r_{s-1}^4 \geq \ell_{s-1}+r_{s-1} = n_0+1$,
    the second inequality follows from $|Q_{s-1}|+r_{s-1}\leq 2\Delta \ell_{s-1}^\frac14 + r_{s-1} \leq 2\Delta (n_0+1)^\frac14+(n_0+1)^\frac14 \leq \frac52\Delta n_0^\frac14$).
    We add the remaining~$r_{s-1}$ vertices (and connecting edges) to~$Q_{s-1}$.\\
    Using this process we construct~$k$ rooted subtrees $(T_1,t_1),\ldots,(T_k,t_k)$.
    Let $T_i=Q_{k+1-i}$ and $t_i:=q_{k+1-i}$ for all $i\in[k]$.
    Obviously, (i) and~(ii) hold, and~(iii) holds by the observation above.
    We also have $n_0^\frac14\leq |T_1| \leq \frac{5}{2}n_0^\frac14$.
    For $i\geq1$, we have $\ell_{k+1-i} = n_0 - n + |Q_1\cup\ldots\cup Q_{k+1-i}| + 1 = n_0 - n + |T_i\cup\ldots\cup T_k| + 1 = n_0 - |T_1\cup\ldots\cup T_{i-1}|+1 = n_{i-1}$.
    As $\ell_{k+1-i}^\frac14 \leq |T_i|\leq \frac52\Delta \ell_{k+1-i}^\frac14$, (iv) holds.
    Hence, $(T_1,t_1),\ldots,(T_{k},t_k)$ satisfy the desired properties.
    }
    \begin{enumerate}[label=\normalfont(\roman*)]
        \item $V(T) = \bigcup_{i=1}^k V(T_i)$,
        \item $T_i\subseteq T(t_i)$ for every $i\in[k]$,
        \item the depth of $t_1,\ldots,t_k$ is non-decreasing, and
        \item $n_{i-1}^\frac14 \leq |T_i|\leq \frac52\Delta n_{i-1}^\frac14$ for every $i\in[k]$.
    \end{enumerate}
\end{lemma}

For $i\in\{2,\ldots,k\}$, let~$w_i$ be the parent of~$t_i$. We add~$w_i$ and $t_iw_i$ to $T_i$ for each $i\in\{2,\ldots,k\}$ and define~$w_i$ as the new root~$t_i$ of~$T_i$. This yields the following result.

\begin{lemma}\label{lemma: tree decomposition}
    Suppose $n,\Delta\geq2$. Let~$(T,r)$ be a rooted tree on~$n$ vertices and $\Delta(T)\leq\Delta$. Let $n_0\in\bN$ such that $n \leq n_0 < n^4$. \textcolor{black}{Then there exist rooted subtrees $(T_1,t_1),\ldots, (T_k,t_k)$} such that, denoting $n_i:=n_0-|T_1\cup\ldots\cup T_i|+1$,
    \begin{enumerate}[label=\normalfont(\roman*)]
        \item $V(T) = \bigcup_{i=1}^k V(T_i)$,
        \item $T_i\subseteq T(t_i)$ for every $i\in[k]$,
        \item the depth of $t_1,\ldots,t_k$ is non-decreasing,
         \item for all $i\in\{2,\ldots,k\}$, there is a unique $j<i$ such that~$T_i$ and~$T_j$ intersect, and in this case we have $V(T_i)\cap V(T_j)=\{t_i\}$, and
        \item $n_{i-1}^\frac14+1 \leq |T_i|\leq 3\Delta n_{i-1}^\frac14$ for every $i\in[k]$.
    \end{enumerate}
\end{lemma}

\subsection{Random trees}
\textcolor{black}{We define a random tree as a random walk that branches at some vertices, and the branches are independent random trees.}
Suppose~$G$ is a graph on~$n$ vertices. Let~$\textbf{x}$ be a perfect fractional matching of~$G$,
let $T=(r_0,\ldots,r_m)$ be a tree
and let $R=(R_0,\ldots,R_m)$ be a random vector whose entries take values in~$V(G)$ and which represents the ordering of~$T$ as follows.
Let~$R_0$ be an arbitrary vertex in~$G$.
\textcolor{black}{We embed~$T$ in a root-to-leaf-manner into~$G$, where the children of any vertex are chosen independently according to the perfect fractional matching of~$G$.
To be more precise, suppose~$r_j$ is a child of~$r_i$. Let $M:=T-(T(r_j)\cup\{r_i\})$. Then~$M$ consists of all vertices whose embedding has no influence on the embedding of~$r_j$, given~$r_i$ has already been embedded. Let~$I$ be the index set of the vertices in~$M$, that is, $I:=\{i\in[m]_0 \mid r_i\in M\}$.
We define $p_{vw}^{ij}:=\pr[R_j=w\mid R_i=v, \bigwedge_{i'\in I'}R_{i'}=v_{i'}]:=\textbf{x}_{vw}$ for all $I'\subseteq I$ if $vw\in E(G)$, and $p_{vw}^{ij}:=0$ otherwise.
Note that, in particular, every path in~$T$ along a root-to-leaf path, when embedded into~$G$, can be regarded as a random walk on~$G$.
We call $R=(R_0,\ldots,R_m)$ a \emph{random tree according to~$T$ and~$\textbf{x}$} in~$G$. We may omit~$\textbf{x}$ if it is clear from the context.}
Observe that we may have $R_i=R_j$ for some $i\neq j$, that is, a realisation of~$R$ in~$G$ might not be a tree.

The definition in case~$T$ is an oriented tree and~$G$ a digraph works analogously, except, provided~$r_j$ is a child of~$r_i$, we define $p_{vw}^{ij}:=\textbf{x}_{vw}$ if $r_ir_j\in E(T)$ and $vw\in E(G)$, and $p_{vw}^{ij}:=\textbf{x}_{wv}$ if $r_jr_i\in E(T)$ and $wv\in E(G)$, and $p_{vw}^{ij}:=0$ otherwise.

\section{Normal perfect fractional matchings with large entropy}\label{sec:Normal perfect fractional matchings with large entropy}

The aim of this section is to prove that any~\mbox{$(n,\eps)$-digraph}~$G$ admits a perfect fractional matching that has large entropy
and is \mbox{$b$-normal} for some not too large~$b$ (Theorem~\ref{theorem: almost maximum matching}).
The latter is needed in Section~\ref{sec:Conditions for a random tree being well-behaved} to show that a random tree is \emph{well-behaved} with high probability.
The fact that the matching has large entropy plays an important role in finding a lower bound on the number of trees of appropriate size in Section~\ref{sec:Counting trees}.
In Lemma~\ref{lemma: new matching} we show that removing a random-like vertex set from~$G$ results in a digraph whose entropy is as large as expected by the natural heuristic. In Lemma~\ref{lemma: matching G-A and G} we also show that removing any not too large vertex set from~$G$ results in a digraph whose entropy is close to~$h(G)$.

We need the following lemmas proven in~\cite{CK:09, GGJKO:21}.

\begin{lemma}[{\cite[Lemma~4.2]{GGJKO:21}}]\label{lemma: c-normal matching}
	Suppose $\frac1n\ll\frac1b\ll\eps$. Let~$G$ be an $(n,\eps)$-digraph.
	Then~$G$ admits a $b$-normal perfect fractional matching.
\end{lemma}

\begin{lemma}[{\cite[Lemma~3.2]{CK:09}}]\label{lemma: 4cycle}
	Let~$G$ be a graph and~$\textbf{x}$ an edge weighting of~$G$.
	Let~$vwuz$ be a 4-cycle in~$G$ where~$\textbf{x}_{vw}\textbf{x}_{uz}\geq\textbf{x}_{wu}\textbf{x}_{zv}$.
	Let~$\tilde{\textbf{x}}\colon E(G)\to\bR_{\geq0}$ be defined by~$\tilde{\textbf{x}}_e:=\textbf{x}_e$ if $e\not\in\{vw,wu,uz,zv\}$ and $\tilde{\textbf{x}}_{vw} := \textbf{x}_{vw}-\alpha$, $\tilde{\textbf{x}}_{uz} := \textbf{x}_{uz}-\alpha$, \textcolor{black}{$\tilde{\textbf{x}}_{wu} := \textbf{x}_{wu}+\alpha$}, $\tilde{\textbf{x}}_{zv} := \textbf{x}_{zv}+\alpha$ for some~$\alpha>0$ such that~$\tilde{\textbf{x}}_{vw}, \tilde{\textbf{x}}_{wu}, \tilde{\textbf{x}}_{uz}, \tilde{\textbf{x}}_{zv}\in[0,1]$.
	If~$\tilde{\textbf{x}}_{vw}\tilde{\textbf{x}}_{uz}\geq\tilde{\textbf{x}}_{wu}\tilde{\textbf{x}}_{zv}$, then~$\tilde{\textbf{x}}$ is an edge weighting of~$G$ with~$h(\tilde{\textbf{x}}) \geq h(\textbf{x})$.
	Additionally, if~$\textbf{x}$ is a perfect fractional matching, so is~$\tilde{\textbf{x}}$.%
\COMMENT{
	Clearly, if~$\sum_{e\in E(v)} \textbf{x}_e=1$ for all~$v\in V(G)$, this is also true for~$\tilde{\textbf{x}}$. Let
	\begin{align*}
		g(\alpha) :=  - (\textbf{x}_{vw}-\alpha)\log(\textbf{x}_{vw}-\alpha) - (\textbf{x}_{wu}+\alpha)\log(\textbf{x}_{wu}+\alpha)
		 - (\textbf{x}_{uz}-\alpha)\log(\textbf{x}_{uz}-\alpha) - (\textbf{x}_{zv}+\alpha)\log(\textbf{x}_{vw}+\alpha).
	\end{align*}
	Then
	\begin{align*}
		g'(\alpha) = \log(\textbf{x}_{vw}-\alpha) - \log(\textbf{x}_{wu}+\alpha) + \log(\textbf{x}_{uz}-\alpha) - \log(\textbf{x}_{zv}+\alpha)
		= \log\left(\frac{(\textbf{x}_{vw}-\alpha)(\textbf{x}_{uz}-\alpha)}{(\textbf{x}_{wu}+\alpha)(\textbf{x}_{zv}+\alpha)}\right).
	\end{align*}
	Thus, as long as $(\textbf{x}_{vw}-\alpha)(\textbf{x}_{uz}-\alpha)\geq(\textbf{x}_{wu}+\alpha)(\textbf{x}_{zv}+\alpha)$, we have~$g'(\alpha)\geq0$ and~$g$ is monotonously nondecreasing.
	Hence,~$h(\tilde{\textbf{x}})-h(\textbf{x}) = g(\alpha)-g(0) \geq 0$.
}
\end{lemma}

The following lemma yields an upper bound on the sum of edge weights of all edges with large edge weight, given the perfect fractional matching has large entropy.

\begin{lemma}\label{lemma: little weight}
	Let~$G$ be a graph on~$2n$ vertices and let $\textbf{x}$ be a perfect fractional matching of~$G$.
	Let~$b>1$ and $M:=\{e\in E(G)\mid \textbf{x}_e\geq\frac{b}{n}\}$.
	Suppose $h(\textbf{x})\geq n\log\frac{n}{2}$.
	Then $\sum_{e\in M}\textbf{x}_e\leq\frac{4n}{\log b}$.
\end{lemma}
\begin{proof}
    Observe the following consequence of Jensen's inequality. If $M\subseteq E(G)$ and \mbox{$\sum_{e\in M}\textbf{x}_e = a$}, then $h_{\textbf{x}}(M)\leq a\log\frac{|M|}{a}$.%
    \COMMENT{
    We have~$\sum_{e\in M}\frac{\textbf{x}_e}{a}=1$. Jensen's inequality yields
    \begin{equation*}
        \frac{h_{\textbf{x}}(M)}{a} = \sum_{e\in M} \frac{\textbf{x}_e}{a}\log\left(\frac{1}{\textbf{x}_e}\right) \leq \log\left(\sum_{e\in M}\frac{1}{a}\right)=\log\frac{\abs{M}}{a}
    \end{equation*}
    and the statement follows from multiplying with~$a$.
    }
    As $\sum_{e\in E(G)}\textbf{x}_e=n$, we have $|M|\leq\frac{n^2}{b}$.
    Let $\beta:=\frac1n\sum_{e\in M}\textbf{x}_e$.
    By the observation stated above, with~$\beta n$ playing the role of~$a$, we obtain $h_{\textbf{x}}(M)\leq \beta n\log\frac{n}{b\beta}$.
    Similarly, $|E(G)\setminus M|\leq 2n^2$ and $\sum_{e\in E(G)\setminus M}\textbf{x}_e=(1-\beta)n$.
    We again use the observation stated above, now with~$(1-\beta)n$ playing the role of~$a$, to obtain $h_{\textbf{x}}(E(G)\setminus M) \leq\left(1-\beta\right)n\log\frac{2n}{1-\beta}$.
    Note that $-x\log x - (1-x)\log(1-x)\leq 1$ for any $x\in[0,1]$, which also follows from Jensen's inequality. In particular, $-\beta n\log\beta-(1-\beta)n\log(1-\beta)\leq n$.
    Thus,
    \begin{align*}
	h(\textbf{x}) &= h_{\textbf{x}}(M) + h_{\textbf{x}}(E(G)\setminus M)
	\leq \beta n\log\frac{n}{b\beta} + (1-\beta)n\log\frac{2n}{1-\beta} \\
	&= \beta n \log\frac{n}{2} + \beta n \log\frac{2}{b\beta} + (1-\beta)n\log \frac{n}{2} + (1-\beta)n \log\frac{4}{1-\beta} \\
	&= n\log\frac{n}{2} - \beta n\log\beta - (1-\beta)n\log(1-\beta) + \beta n + 2(1-\beta)n - \beta n\log b \\
	&\leq n\log\frac{n}{2} + 4n - \beta n\log b.
    \end{align*}
    As $h(\textbf{x})\geq n\log\frac{n}{2}$ by assumption, we obtain $\beta \leq \frac{4}{\log b}$.
\end{proof}

The following result shows that we can modify a perfect fractional matching with large entropy to obtain a perfect fractional matching which is $b$-normal and whose entropy is only slightly smaller. Note that any digraph~$G$ on~$n$ vertices with $\delta^0(G)\geq\frac{n}{2}$ admits a perfect fractional matching~$\textbf{m}$ that satisfies $h(\textbf{m})\geq n\log\delta^0(G)$ (\cite[Theorem~3.1]{CK:09}).

\begin{theorem} \label{theorem: almost maximum matching}
    Suppose $\frac1n\ll\frac1b\ll\eps$. Let~$G$ be an \mbox{$(n,\eps)$-digraph} and let~$\textbf{m}$ be a perfect fractional matching of~$G$ that satisfies $h(\textbf{m})\geq n\log\frac{n}{2}+\eps n$. Then there is a $b$-normal perfect fractional matching $\textbf{x}$ of~$G$ that satisfies $h(\textbf{x}) \geq h(\textbf{m}) - \eps n$.
\end{theorem}
\begin{proof}
    We first prove that~$G$ admits a matching~$\textbf{x}$ that satisfies $h(\textbf{x})\geq h(\textbf{m})-\eps n$ and where each edge weight is not too small by taking a linear combination of~$\textbf{m}$ and a $c$-normal perfect fractional matching. We then show that it is possible to reduce the weight of edges with large weight without reducing the entropy.
    
    Let~$\lambda$ and~$c$ be such that $\frac{1}{b}\ll\lambda\ll\frac{1}{c}\ll\eps$.
    Let~$\textbf{n}$ be a $c$-normal perfect fractional matching of~$G$ (which exists by Lemma~\ref{lemma: c-normal matching}). 
    Define \mbox{$\textbf{x}:=(1-\lambda)\textbf{m} + \lambda\textbf{n}$}.
    Then~$\textbf{x}$ is a perfect fractional matching and
    \begin{equation}\label{h(x)}
        h(\textbf{x}) = (1-\lambda) \sum_{e\in E(G)} \textbf{m}_e\log\frac{1}{\textbf{x}_e} + \lambda \sum_{e\in E(G)} \textbf{n}_e\log\frac{1}{\textbf{x}_e}.
    \end{equation}

    In what follows we bound both terms in~\eqref{h(x)} individually.
    Observe that $\log\frac{1}{\textbf{x}_e} = \log\frac{1}{\textbf{m}_e} + \log\frac{1}{1-\lambda+\lambda\frac{\textbf{n}_e}{\textbf{m}_e}} \geq \log\frac{1}{\textbf{m}_e} - \log(1+\lambda\frac{\textbf{n}_e}{\textbf{m}_e}) \geq \log\frac{1}{\textbf{m}_e}-2\lambda\frac{\textbf{n}_e}{\textbf{m}_e}$. This yields
    \begin{align}\label{m_e}
        \sum_{e\in E(G)} \textbf{m}_e\log\frac{1}{\textbf{x}_e}
        &\geq h(\textbf{m}) - 2\lambda \sum_{e\in E(G)} \textbf{n}_e
        = h(\textbf{m}) - 2\lambda n,
    \end{align}
    which yields a lower bound for the first term in~\eqref{h(x)}.
    
    To bound the second term in~\eqref{h(x)} from below, we proceed as follows. For $i\in\bZ$, let $M_i:=\{e\in E(G)\mid \frac{c^{i-1}}{n} \leq \textbf{m}_e < \frac{c^i}{n}\}$.
    Note that $|M_i| \leq \frac{n^2}{c^{i-1}}$.
    Define further $N_-:=\bigcup_{i<0} M_i$, $N_0:=M_0\cup M_1$ and $N_+:=\bigcup_{i\geq 2} M_i$.

    Observe that $\textbf{n}_e\geq\frac{1}{cn}$ and thus $\textbf{n}_e\log\frac{1}{\textbf{n}_e}\leq\textbf{n}_e\log(cn)$. This implies
    \begin{equation}\label{n_elogn}
        \textbf{n}_e\log n \geq \textbf{n}_e\log\frac{1}{\textbf{n}_e} - \textbf{n}_e\log c.
    \end{equation}
    
    For $e\in N_-$, we have $\textbf{x}_e\leq\textbf{n}_e$ and thus
    \begin{equation}\label{N_-}
        \sum_{e\in N_-} \textbf{n}_e\log\frac{1}{\textbf{x}_e} 
        \geq \sum_{e\in N_-} \textbf{n}_e\log\frac{1}{\textbf{n}_e}.
    \end{equation}

    For $e\in N_0$, we have $\textbf{x}_e\leq\frac{c}{n}$. This yields
    \begin{align}\label{N_0}
        \sum_{e\in N_0} \textbf{n}_e \log\frac{1}{\textbf{x}_e}
        &\geq \sum_{e\in N_0} \textbf{n}_e \log\frac{n}{c}
        \stackrel{\eqref{n_elogn}}{\geq} \sum_{e\in N_0} \textbf{n}_e\log\frac{1}{\textbf{n}_e} - 2\log c\cdot \sum_{e\in N_0} \textbf{n}_e \\
        &\geq \sum_{e\in N_0} \textbf{n}_e\log\frac{1}{\textbf{n}_e} - 2n\log c. \nonumber
    \end{align}

    For $e\in N_+$, recall that $\textbf{x}_e\leq\textbf{m}_e$. Note that $\sum_{e\in M_i}\textbf{n}_e \leq |M_i|\frac{c}{n} \leq nc^{2-i}$ for $i\geq 2$. We compute
    \begin{align*}
        \sum_{e\in M_i} \textbf{n}_e\log\frac{1}{\textbf{x}_e}
        &\geq \sum_{e\in M_i} \textbf{n}_e\log\frac{1}{\textbf{m}_e}
        \geq \sum_{e\in M_i} \textbf{n}_e\log\frac{n}{c^i}
        = \sum_{e\in M_i} \textbf{n}_e\log n - i\log c\cdot\sum_{e\in M_i}\textbf{n}_e \\
        &\stackrel{\eqref{n_elogn}}{\geq} \sum_{e\in M_i}\textbf{n}_e\log\frac{1}{\textbf{n}_e} - 2i\log c\cdot\sum_{e\in M_i}\textbf{n}_e
        \geq \sum_{e\in M_i}\textbf{n}_e\log\frac{1}{\textbf{n}_e} - 2n\log c\cdot ic^{2-i}.
     \end{align*}
     This yields%
    \COMMENT{We have
    \begin{align*}
        2 + \sum_{i\geq 3} \frac{i}{c^{i-2}}
        = \sum_{i\geq 2} \frac{i}{c^{i-2}}
        = \sum_{i\geq 3} \frac{i}{c^{i-3}} - \sum_{i\geq 3} \frac{1}{c^{i-3}}
        = c\sum_{i\geq 3} \frac{i}{c^{i-2}} - \sum_{i\geq 0} \frac{1}{c^i}
        = c\sum_{i\geq 3} \frac{i}{c^{i-2}} - \frac{c}{c-1}
    \end{align*}
    and thus $\sum_{i\geq3}\frac{i}{c^{i-2}} = \frac{2+\frac{c}{c-1}}{c-1} \leq 1$.
    This yields $\sum_{i\geq 2} \frac{i}{c^{i-2}}
        = 2 + \sum_{i\geq 3} \frac{i}{c^{i-2}}
        \leq 3$.}
     \begin{align}\label{N_+}
        \sum_{e\in N_+} \textbf{n}_e\log\frac{1}{\textbf{x}_e}
        \geq \sum_{e\in N_+} \textbf{n}_e\log\frac{1}{\textbf{n}_e} - 2n\log c\sum_{i\geq 2}\frac{i}{c^{i-2}}
        \geq \sum_{e\in N_+} \textbf{n}_e\log\frac{1}{\textbf{n}_e} - 6n\log c,
    \end{align}
    and hence by employing \eqref{N_-}, \eqref{N_0} and \eqref{N_+}, we conclude
    \begin{align} \label{n_e}
        \sum_{e\in E(G)} \textbf{n}_e\log\frac{1}{\textbf{x}_e} \geq h(\textbf{n}) - 8n\log c,
    \end{align}
    which is a lower bound for the second term in \eqref{h(x)}.
    
    Consequently, using \eqref{h(x)}, \eqref{m_e} and \eqref{n_e} gives rise to
    \begin{align*}
        h(\textbf{x})
        \geq (1-\lambda) h(\textbf{m}) + \lambda h(\textbf{n}) - 9\lambda n \log c.
    \end{align*}
    In particular, as $h(\textbf{n}) \geq n\log n - n\log c$ and $h(\textbf{m})\leq n\log n$, this yields
    \begin{align*}
        h(\textbf{x})
        \geq h(\textbf{m}) - 10\lambda n\log c
        \geq h(\textbf{m}) - \eps n,
    \end{align*}
    where we used $\lambda\ll\frac1c\ll\eps$ in the last inequality.
    In particular, we have $h(\textbf{x})\geq n\log\frac{n}{2}$.

    Observe that $\textbf{x}_e\geq\frac{\lambda}{cn}$ for all $e\in E(G)$.
    We prove that it is possible to reduce the weight of the edges with large weight
	without reducing the entropy of~$\textbf{x}$
	and without changing the lower bound on~$\textbf{x}$ significantly.
    
    These arguments are easier to comprehend if we look at the balanced bipartite graph on~$2n$ vertices corresponding to~$G$, which we call~$B$. We also write~$\textbf{x}$ for the perfect fractional matching of~$B$ that arises naturally from that of~$G$.
    Note that these edge weightings are essentially the same, and if we prove that we can change the edge weighting on~$B$ to obtain a $b$-normal perfect fractional matching, the same holds true for~$G$.

    Let $M:=\{e\in E(B) \mid \textbf{x}_e\geq\frac{b}{n}\}$
	(we may assume that~$M$ is nonempty).
	Fix~$vw\in M$.
	We consider 4-cycles~$vwuz$ in~$B$ where $\textbf{x}_{wu},\textbf{x}_{vz}\leq\frac{1}{\eps n}$.
	We claim that there are (at least)~$\eps^2 n^2$ such 4-cycles.
    Let $Z:=\{z\in N(v)\mid \textbf{x}_{vz}\leq \frac{1}{\eps n}\}$ and $U:=\{u\in N(w) \mid \textbf{x}_{wu} \leq \frac{1}{\eps n}\}$.
	As there are at most~$\eps n$ choices to select $x\in N(v)$ such that $\textbf{x}_{vx}\geq\frac{1}{\eps n}$, we obtain $|Z|\geq\delta(B)-\eps n\geq \frac{n}{2}$, and by similar reasoning $|U|\geq\frac{n}{2}$.
    As $\delta(B)\geq(\frac12+\eps)n$, each vertex $z\in Z$ has at least~$\eps n$ neighbours in~$U$. Let $Z'\subseteq Z$ such that $|Z'|=\eps n$, and for each $z\in Z'$, let $U_z'\subseteq U\cap N(z)$ such that $|U'_z|=\eps n$. 
    Let $U':=\bigcup_{z\in Z'}U_z'$.
    As for $vwuz$ is a 4-cycle for any $z\in Z'$ and $u\in U_z'$, 
    this yields~$\eps^2 n^2$ 4-cycles $vwuz$ where $\textbf{x}_{wu},\textbf{x}_{vz}\leq\frac{1}{\eps n}$ (note that the actual number of such 4-cycles may be larger).
	
	Let~$vwuz$ be such a 4-cycle.
    Then $\textbf{x}_{vw}\textbf{x}_{zu} \geq \frac{b}{n}\frac{\lambda}{cn} \geq \frac{1}{\eps^2n^2} \geq \textbf{x}_{vz}\textbf{x}_{wu}$, as $\frac{1}{b}\ll\lambda\ll\frac{1}{c}\ll\eps$.
	According to Lemma~\ref{lemma: 4cycle},
	we can reduce the weights of~$\textbf{x}_{vw}$ and~$\textbf{x}_{zu}$
	and increase those of~$\textbf{x}_{vz}$ and~$\textbf{x}_{wu}$ by the same amount
	as long as $\tilde{\textbf{x}}_{vw}\tilde{\textbf{x}}_{zu} \geq \tilde{\textbf{x}}_{vz}\tilde{\textbf{x}}_{wu}$
	without reducing the entropy, where~$\tilde{\textbf{x}}$ denotes the matching that arises from this process.
    As we show below, for these edges this inequality is true with much room to spare.
    
	Let~$r_{vw} := \frac{\textbf{x}_{vw}}{\eps^2 n^2}$.
	By redistributing weight~$r_{vw}$ as described above on the chosen 4-cycles,
    we can reduce~$\textbf{x}_{vw}$ to any arbitrary lower nonnegative weight.
    We reduce~$\textbf{x}_{vw}$ until $\textbf{x}_{vw}\leq\frac{b}{n}$.
    Recall that $|U'_z|=\eps n$ and thus~$\textbf{x}_{vz}$ is increased~$\eps n$ times for any $z\in Z'$. Note further that $|Z'|=\eps n$ and thus~$\textbf{x}_{wu}$ is increased at most~$\eps n$ times for any $u\in U'$.
	Thus, $\tilde{\textbf{x}}_{vz},\tilde{\textbf{x}}_{wu}\leq\frac{1}{\eps n}+\eps n r_{vw} \leq \frac{2}{\eps n}$. To such an edge we do not add weight in the process described above. Thus, no new edges with weight at least~$\frac{b}{n}$ arise.
 
	We turn to the next edge in~$M$ and continue the process described above until every edge $e\in E(B)$ satisfies $\tilde{\textbf{x}}_e\leq\frac{b}{n}$.
    According to Lemma~\ref{lemma: little weight}, we have $\sum_{e\in M}\textbf{x}_e \leq \frac{4n}{\log b}$.
    Thus, the maximum amount of weight by which an edge opposite to an edge in~$M$ is reduced is $\sum_{e\in M}r_e \leq \frac{4}{\eps^2n\log b}$.
    Thus, $\tilde{\textbf{x}}_{zu}\geq\frac{\lambda}{cn}-\frac{4}{\eps^2 n\log b} \geq \frac{\lambda^2}{n}$, as $\frac1b\ll\lambda\ll\frac1c\ll\eps$.

    Altogether, $\tilde{\textbf{x}}_{vw}\tilde{\textbf{x}}_{zu} \geq \tilde{\textbf{x}}_{vz}\tilde{\textbf{x}}_{wu}$ holds if $\tilde{\textbf{x}}_{vw}\frac{\lambda^2}{n} \geq (\frac{2}{\eps n})^2$, that is, if $\tilde{\textbf{x}}_{vw}\geq\frac{4}{\lambda^2\eps^2n}$.
    In particular, as $\frac{1}{b}\ll\lambda\ll\frac{1}{c}\ll\eps$, we can reduce weights such that $\tilde{\textbf{x}}_{vw} \leq \frac{b}{n}$ for each $vw\in M$.
    Thus, $\tilde{\textbf{x}}$ is \mbox{$b$-normal}.
\end{proof}

Suppose we remove a random-like set $M\subseteq V$ from an $(n,\eps)$-digraph with a $b$-normal perfect fractional matching. The following lemma essentially shows that~$G-M$ admits a $b'$-normal perfect fractional matching that only loses the expected amount of entropy and where~$b'$ is only slightly larger than~$b$.
\begin{lemma}\label{lemma: new matching}
    Suppose $\frac1n\ll\gamma\ll\frac1b\ll\eps$. Let~$G$ be an $(n,\eps)$-digraph that is an induced subgraph of some digraph~$F$.
    Let $u\in V(F)\setminus V(G)$ and
    let~$\textbf{x}$ be a $b$-normal perfect fractional matching of~$G$.
    Let $\Delta:= e^{\gamma\sqrt{\ln n}}$, let $n^\frac14\leq m\leq 3\Delta n^\frac14$, and let $M\subseteq V(G)$ be of size~$m+1$.
    Suppose that for all $v\in V(G)$ and $*\in\{+,-\}$ the following holds.
    \begin{enumerate}[label=\normalfont(\roman*)]
        \item $\sum_{w\in M\cap N_G^{+}(v)}\textbf{x}_{vw} = \frac{m+1}{n} \pm n^{-\frac34-\frac{1}{18\sqrt{\ln n}}}$
        and $\sum_{w\in M\cap N_G^{-}(v)}\textbf{x}_{wv} = \frac{m+1}{n} \pm n^{-\frac34-\frac{1}{18\sqrt{\ln n}}}$
        \item $\sum_{w\in M\cap N_G^{+}(v)}\textbf{x}_{vw}\log\frac{1}{\textbf{x}_{vw}} = \frac{m+1}{n}h_{\textbf{x}}^{+}(v) \pm n^{-\frac34-\frac{1}{18\sqrt{\ln n}}}$
        and \\$\sum_{w\in M\cap N_G^{-}(v)}\textbf{x}_{wv}\log\frac{1}{\textbf{x}_{wv}} = \frac{m+1}{n}h_{\textbf{x}}^{-}(v) \pm n^{-\frac34-\frac{1}{18\sqrt{\ln n}}}$
        \item $|N_{F,G}^{*}(u)|\geq (\frac12+\eps)n$ and
        \item $|N_{F,G}^{*}(u)\cap M|=\frac{m+1}{n}|N_{F,G}^{*}(u)|\pm n^{\frac14-\frac{1}{17\sqrt{\ln n}}}$.
    \end{enumerate}
    Then there is a $(1+n^{-\frac34-\frac{1}{22\sqrt{\ln n}}})b$-normal perfect fractional matching~$\textbf{z}$ of $F[(V(G)\setminus M)\cup\{u\}]$ that satisfies $h(\textbf{z}) \geq \frac{n-m}{n}h(\textbf{x}) - (n-m)\log\frac{n}{n-m} - n^{\frac14-\frac{1}{24\sqrt{\ln n}}}$.
\end{lemma}
\begin{proof}
    Let $G':=F[(V(G)\setminus M)\cup\{u\}]$, let $n':=|V(G')|=n-m$, and let $B':=B_{G'}$ be the balanced bipartite graph on~$2n'$ vertices corresponding to~$G'$. We define an edge weighting~$\textbf{x}'$ on~$G'$ (and thus on~$B'$) by $\textbf{x}'_e:=\textbf{x}_e$ if $e\in E(G)$ and $\textbf{x}'_{uv}:=\frac{1}{|N_{F,G}^{+}(u)|}$ for any $v\in N_{F,G}^{+}(u)$ and $\textbf{x}'_{vu}:=\frac{1}{|N_{F,G}^{-}(u)|}$ for any $v\in N_{F,G}^{-}(u)$.
    
    Let us refer to the weight of a vertex as the sum of the weight of incident edges. Our proof strategy is as follows: We rescale the weights of the edges of~$B'$ such that they sum up to~$n'$, which is exactly the total weight of a perfect fractional matching of~$B'$. Then we redistribute weight from vertices with too much weight to vertices with too little weight, using that any two vertices in the same partition class of~$B'$ share many neighbours.
    In both steps, the entropy of the matching does not deviate too much from $\frac{n'}{n}h(\textbf{x})$.
    
    Define $\lambda:=\frac{n'}{\sum_{e\in E(B')}\textbf{x}'_e}$ and $\textbf{y}_e:=\lambda\textbf{x}'_e$ for $e\in E(B')$. Then $\sum_{e\in E(B')} \textbf{y}_e=n'$.
    Let $E(M,G):=\{vw\in E(G)\mid \text{$v\in M$ or $w\in M$}\}$.
    Note that
    \begin{align}\label{equation: sum x_e'}
        \sum_{e\in E(B')}\textbf{x}'_e = \sum_{e\in E(G)}\textbf{x}_e - \sum_{e\in E(M,G)} \textbf{x}_e + \sum_{v\in V(G)\setminus M}\textbf{x}'_{vu} + \textbf{x}'_{uv}.
    \end{align}
    By using (i) we obtain
    \begin{align*}
        \sum_{e\in E(M,G)} \textbf{x}_e
        &= \sum_{v\in M} \left(\sum_{w\in (V(G)\setminus M) \cap N_G^{+}(v)} \textbf{x}_{vw} + \sum_{w\in (V(G)\setminus M) \cap N_G^{-}(v)} \textbf{x}_{wv} + \sum_{w\in M\cap N_G^+(v)} \textbf{x}_{vw}\right) \\
        &= \sum_{v\in M}\left(1 + \sum_{w\in (V(G)\setminus M) \cap N_G^{-}(v)} \textbf{x}_{wv}\right)
        = (m+1)\left(2 - \frac{m+1}{n} \pm n^{-\frac34-\frac{1}{18\sqrt{\ln n}}}\right) \\
        &= 2(m+1) - \frac{(m+1)^2}{n} \pm mn^{-\frac34-\frac{1}{20\sqrt{\ln n}}}.
    \end{align*}
    Using $(V(G)\setminus M)\cap N^*_{F,G}(u) = N_{F,G}^*(u)\setminus(N_{F,G}^*(u)\cap M)$ for each $*\in\{+,-\}$ and applying~(iii) and~(iv) yields
    \begin{align*}
        \sum_{v\in V(G)\setminus M}\textbf{x}'_{uv}+\textbf{x}'_{vu} &= \frac{|(V(G)\setminus M)\cap N_{F,G}^{+}(u)|}{|N_{F,G}^{+}(u)|} + \frac{|(V(G)\setminus M)\cap N_{F,G}^{-}(u)|}{|N_{F,G}^{-}(u)|} \\
        &= 2\frac{n-(m+1)}{n} \pm \frac{n^{\frac14-\frac{1}{17\sqrt{\ln n}}}}{|N_{F,G}^{+}(u)|} \pm \frac{n^{\frac14-\frac{1}{17\sqrt{\ln n}}}}{|N_{F,G}^{-}(u)|} \\
        &= 2\frac{n-(m+1)}{n} \pm n^{-\frac34-\frac{1}{20\sqrt{\ln n}}},
    \end{align*}
    where the last inequality follows from $|N_{F,G}^+(u)|,|N_{F,G}^-(u)|\geq\frac{n}{2}$.
    Note that $\frac{m^2-1}{n} \geq (m+1)n^{-\frac34-\frac{1}{20\sqrt{\ln n}}}$, which is easily seen to be true by multiplying both sides of the inequality with~$n$ and using $m\geq n^\frac14$. Thus, \eqref{equation: sum x_e'} implies that
    \begin{align*}
        \sum_{e\in E(B')}\textbf{x}'_e &\geq n - 2(m+1) + \frac{(m+1)^2}{n} + 2\frac{n-(m+1)}{n} - (m+1)n^{-\frac34-\frac{1}{20\sqrt{\ln n}}}\\
        &= n - 2m + \frac{m^2-1}{n} - (m+1)n^{-\frac34-\frac{1}{20\sqrt{\ln n}}}
        \geq n-2m
    \end{align*}
    and
    \begin{align*}
        \sum_{e\in E(B')}\textbf{x}'_e &\leq n - 2(m+1) + \frac{(m+1)^2}{n} + 2\frac{n-(m+1)}{n} + (m+1)n^{-\frac34-\frac{1}{20\sqrt{\ln n}}} \\
        &= n-2m+\frac{m^2-1}{n} + (m+1)n^{-\frac34-\frac{1}{20\sqrt{\ln n}}}
        \leq n-2m+\frac{2m^2}{n}.
    \end{align*}
    Thus we have
    \begin{equation}\label{equation: bounds for lambda}
        1 + \frac{m-\frac{2m^2}{n}}{n-2m+\frac{2m^2}{n}} = \frac{n-m}{n-2m+\frac{2m^2}{n}} \leq \lambda \leq \frac{n-m}{n-2m} = 1 + \frac{m}{n-2m}.
    \end{equation}

    In what follows, we redistribute weights of~$\textbf{y}$ such that the edge weighting~$\textbf{z}$ which arises from this process is a perfect fractional matching.

    We proceed by induction on the number of vertices $v\in V(B')$ where $\sum_{v'\in N_{B'}(v)}\textbf{y}_{vv'}\neq 1$.
    If there is no such vertex, then~$\textbf{y}$ is a perfect fractional matching of~$B'$.
    Otherwise, there is a vertex $v\in V(B')$ with $\sum_{v'\in N_{B'}(v)}\textbf{y}_{vv'} = 1 + x_1$ and $x_1>0$ and a vertex $w\in V(B')$ in the same partition class as~$v$ such that $\sum_{w'\in N_{B'}(w)}\textbf{y}_{ww'} = 1 - x_2$ where $x_2 > 0$.\COMMENT{As the average weight is 1.} 
    \textcolor{black}{Observe that~$B'$ satisfies $\delta(B')\geq (\frac12+\frac{\eps}{2})n' \geq \frac{n'}{2}+\frac{\eps n}{4}$. In particular, we have $|N_{B'}(v)\cap N_{B'}(w)|\geq \frac{\eps n}{2}$.
    For each $z\in N_{B'}(v)\cap N_{B'}(w)$, we can redistribute weight from $\sum_{v'\in N_{B'}(v)}\textbf{y}_{vv'}$ to $\sum_{w'\in N_{B'}(w)}\textbf{y}_{ww'}$ by adding it to $\textbf{y}_{zw}$ and subtracting it from~$\textbf{y}_{vz}$. In this case we say we shift weight from~$vz$ to~$zw$.
    We distribute weight $\min\{x_1,x_2\}$ on paths of length~2 between~$v$ and~$w$
    Let $Z\subseteq N_{B'}(v)\cap N_{B'}(w)$ such that $|Z|=\frac{\eps n}{2}$.}
    For each $z\in Z$, we add weight~$\frac{2\min\{x_1,x_2\}}{\eps n}$ to~$\textbf{y}_{zw}$ and subtract it from~$\textbf{y}_{vz}$ (we will later see that this change is very small and it is indeed possible to distribute weight like this).
    In this case we say that~$zw$ gets weight from~$v$ and~$w$ and~$vz$ loses weight from~$v$ and~$w$.
    Call the edge weighting which arises from this process~$\tilde{\textbf{y}}$.
    Now we either have $\sum_{v'\in N_{B'}(v)}\tilde{\textbf{y}}_{vv'} = 1$ or $\sum_{w'\in N_{B'}(w)}\tilde{\textbf{y}}_{ww'} = 1$ and thus $|\{v\in V(B')\mid \sum_{v'\in N_{B'}(v)}\tilde{\textbf{y}}_{vv'}\neq 1| \leq |\{v\in V(B')\mid \sum_{v'\in N_{B'}(v)}{\textbf{y}}_{vv'}\neq 1|-1$.
    Using our induction hypothesis, we can turn~$\tilde{\textbf{y}}$ into a perfect fractional matching~$\textbf{z}$ such that $\sum_{v'\in N_{B'}(v)} \textbf{z}_{vv'} = 1$ for each $v\in V(B')$.
        
    Assume $vz\in E(B')$ is used in the process above to distribute edge weights. Note that $\textbf{y}_{vz}$ is only changed as long as \mbox{$\sum_{v'\in N_{B'}(v)} \textbf{y}_{vv'}\neq1$} or $\sum_{z'\in N_{B'}(z)}\textbf{y}_{zz'}\neq1$.
    Suppose $\sum_{v'\in N_{B'}(v)}\textbf{y}_{vv'}=1\pm x_1$ for some $x_1>0$, $\sum_{w'\in N_{B'}(w)}\textbf{y}_{ww'}=1\pm x_2$ for some $x_2>0$, and we distribute weight from~$v$ to~$w$. If $x_1<x_2$, we obtain $\tilde{\textbf{y}}_{vz}=\textbf{y}_{vz}\pm \frac{2x_1}{\eps n}$ for any $z\in Z$. If $x_1\geq x_2$, we obtain $\tilde{\textbf{y}}_{vz}=\textbf{y}_{vz}\pm \frac{2x_2}{\eps n}$ for any $z\in Z$, and we have $\sum_{v'\in N_{B'}(v)}\tilde{\textbf{y}}_{vv'} = 1\pm (x_1-x_2)$.
    Thus the maximum amount of weight that~$vz$ can obtain from~$v$ is~$\frac{2x_1}{\eps n}$.
    Hence, denoting $y_1:=|\sum_{z'\in N_{B'}(z)}\textbf{y}_{zz'}-1|$, 
    we have $\textbf{z}_{vz}=\textbf{y}_{vz}\pm 2\cdot \frac{2\max\{x_1,y_1\}}{\eps n}$.
    
    \textcolor{black}{Let $x:=\max_{v\in V(B')}|\sum_{v'\in N_{B'}(v)}\textbf{y}_{vv'}-1|$. We bound~$x$ from above. Let $v\in V(G)\setminus M$ and $v^+,v^-$ be its corresponding vertices in~$B'$.
    Note that
    \begin{align*}
        1 = \sum_{w\in N_G^+(v)\cap M} \textbf{x}_{vw} + \sum_{w\in N_G^+(v)\setminus M} \textbf{x}_{vw} = \sum_{w\in N_G^+(v)\cap M} \textbf{x}_{vw} + \sum_{v'\in N_{B'}(v^+)} \textbf{x}'_{v^+v'} - \textbf{x}'_{vu}.
    \end{align*}
    By (i) we have
    \begin{align*}
        \sum_{v'\in N_{B'}(v^+)} \textbf{x}'_{v^+v'} = 1 - \sum_{w\in N_G^{+}(v)\cap M} \textbf{x}_{vw} + \textbf{x}'_{vu} \leq 1-\frac{m+1}{n}+n^{-\frac34-\frac{2}{37\sqrt{\ln n}}}.
    \end{align*}
    Using~\eqref{equation: bounds for lambda}, this yields
    \begin{align*}
        \sum_{v'\in N_{B'}(v^+)}\textbf{y}_{v^+v'} = \lambda\sum_{v'\in N_{B'}(v^+)} \textbf{x}'_{v^+v'} \leq 1 + \frac{2m^2}{n(n-2m)} + 2n^{-\frac34-\frac{2}{37\sqrt{\ln n}}} \leq 1 + n^{-\frac34-\frac{1}{19\sqrt{\ln n}}},
    \end{align*}
    and the analogous inequality holds for~$v^-$.%
    \COMMENT{
            we have
            \begin{align*}
                \lambda\sum_{v'\in N_{B'}(v^+)} \textbf{x}'_{v^+v'}
                &\leq \left(1+\frac{m}{n-2m}\right)\left(1-\frac{m+1}{n}+n^{-\frac34-\frac{2}{37\sqrt{\ln n}}}\right) \\
                &= 1-\frac{m+1}{n}+n^{-\frac34-\frac{2}{37\sqrt{\ln n}}} + \frac{m}{n-2m} - \frac{m+1}{n}\frac{m}{n-2m} + \frac{m}{n-2m}n^{-\frac34-\frac{2}{37\sqrt{\ln n}}} \\
                &\leq 1 -\frac{m}{n} + \frac{m}{n-2m} + 2n^{-\frac34-\frac{2}{37\sqrt{\ln n}}}
                = 1 + \frac{2m^2}{n(n-2m)} + 2n^{-\frac34-\frac{2}{37\sqrt{\ln n}}} \\
                &\leq 1 + 3n^{-\frac34-\frac{2}{37\sqrt{\ln n}}}
                \leq 1 + n^{-\frac34-\frac{1}{19\sqrt{\ln n}}},
            \end{align*}
            (if we subtract less in the exponent, the term gets larger)
    }
    Similarly we obtain
    \begin{align*}
        \sum_{v'\in N_{B'}(v^+)}\textbf{y}_{v^+v'} \geq 1 - n^{-\frac34-\frac{1}{19\sqrt{\ln n}}}
    \end{align*}
    and the analogous inequality for~$v^-$.%
    \COMMENT{
            we have
            \begin{align*}
                \lambda \sum_{v'\in N_{B'}(v^+)} \textbf{x}'_{v^+v'}
                &\geq \left(1+\frac{m-\frac{2m^2}{n}}{n-2m+\frac{2m^2}{n}}\right) \left(1-\frac{m+1}{n}-n^{-\frac34-\frac{2}{37\sqrt{\ln n}}}\right) \\
                &\geq 1 - \frac{m+1}{n} - 2n^{-\frac34-\frac{2}{37\sqrt{\ln n}}} + \frac{m-\frac{2m^2}{n}}{n-2m+\frac{2m^2}{n}} - \frac{m+1}{n}\frac{m-\frac{2m^2}{n}}{n-2m+\frac{2m^2}{n}} \\
                &\geq 1 - 3n^{-\frac34-\frac{2}{37\sqrt{\ln n}}} - \frac{(m+1)n-2m(m+1)+\frac{2m^2(m+1)}{n} - nm + 2m^2}{n^2-2mn+2m^2} \\
                &= 1 - 3n^{-\frac34-\frac{2}{37\sqrt{\ln n}}} - \frac{n-2m+\frac{2m^3+2m^2}{n}}{n^2-2mn+2m^2}
                \geq 1 - 4n^{-\frac34-\frac{2}{37\sqrt{\ln n}}}
                \geq 1 - n^{-\frac34-\frac{1}{19\sqrt{\ln n}}}.
            \end{align*}
    }
    Further, by~(iv) we have
    \begin{align*}
        \sum_{v'\in N_{B'}(u^+)}\textbf{x}'_{u^+v'} = \frac{1}{|N_{F,G}^+(u)|}\left(\frac{n-(m+1)}{n}|N_{F,G}^+(u)| \pm n^{\frac14-\frac{1}{17\sqrt{\ln n}}}\right) = 1 - \frac{m+1}{n} \pm n^{-\frac34-\frac{1}{18\sqrt{\ln n}}}.
    \end{align*}
    As in the calculations above, we obtain $\sum_{v'\in N_{B'}(u^+)}\textbf{y}_{u^+v'} = 1 \pm n^{-\frac34-\frac{1}{19\sqrt{\ln n}}}$ and the same inequality holds for~$u^-$.
    Thus, for each $v\in V(B')$, we need to redistribute at most weight $n^{-\frac34-\frac{1}{19\sqrt{\ln n}}}$ to achieve $\sum_{v'\in N_{B'}(v)}\textbf{y}_{vv'}=1$, and every edge weight is changed by at most~$4\eps^{-1}n^{-\frac74-\frac{1}{19\sqrt{\ln n}}} \leq n^{-\frac74-\frac{1}{20\sqrt{\ln n}}}$.} Hence,
    \begin{equation}\label{equation: z_e}
        \textbf{z}_e = \textbf{y}_e \pm n^{-\frac74-\frac{1}{20\sqrt{\ln n}}}
    \end{equation}
    for each $e\in E(B')$.

    We now bound the change of entropy in this process from above. Assume that~$e,e'\in E(B')$ are used to redistribute weight, and assume that weight~$\alpha\leq n^{-\frac74-\frac{1}{20\sqrt{\ln n}}}$ is subtracted from~$\textbf{y}_e$ and added to~$\textbf{y}_{e'}$.
    Define $\mu_1:=1-\frac{\alpha}{\textbf{y}_e}$ and $\mu_2:=1+\frac{\alpha}{\textbf{y}_{e'}}$.
    Then $\textbf{y}_e\log\frac{1}{\mu_1}\geq 0$, and, using $\log (1+x) \leq 2x$ for $x\geq0$, we obtain $\textbf{y}_{e'}\log\frac{1}{\mu_2} = -\textbf{y}_{e'}\log\mu_2 \geq -2\alpha$. As $\textbf{y}_e-\alpha\geq\frac{\textbf{y}_e}{2}$ and $\textbf{y}_{e'}+\alpha \leq 2\textbf{y}_{e'}$, we also have
    \begin{align*}
        -\left(\log\frac{1}{\textbf{y}_e-\alpha}-\log\frac{1}{\textbf{y}_{e'}+\alpha}\right)
        &\geq\log\textbf{y}_e - \log\textbf{y}_{e'} - 2
        \geq \log \frac{\lambda}{bn} - \log\frac{\lambda b}{n} - 2
        \geq -3\log b.
    \end{align*}
    This yields $\textbf{y}_e\log\frac{1}{\mu_1} + \textbf{y}_{e'}\log\frac{1}{\mu_2} - \alpha(\log\frac{1}{\textbf{y}_e-\alpha}-\log\frac{1}{\textbf{y}_{e'}+\alpha}) \geq -5\alpha \log b \geq -n^{-\frac74-\frac{1}{21\sqrt{\ln n}}}$.
    The entropy of $e,e'$ hence satisfies
    \begin{align*}
        h_{\textbf{z}}(\{e, e'\})
        &= (\textbf{y}_e-\alpha)\log\frac{1}{\textbf{y}_e-\alpha} + (\textbf{y}_{e'}+\alpha)\log\frac{1}{\textbf{y}_{e'}+\alpha} \\
        &= \textbf{y}_e\log\frac{1}{\mu_1\textbf{y}_e} + \textbf{y}_{e'}\log\frac{1}{\mu_2\textbf{y}_{e'}} - \alpha\left(\log\frac{1}{\textbf{y}_e-\alpha}-\log\frac{1}{\textbf{y}_{e'}+\alpha}\right) \\
        &\geq \textbf{y}_e\log\frac{1}{\textbf{y}_e} + \textbf{y}_{e'}\log\frac{1}{\textbf{y}_{e'}}
        -n^{-\frac74-\frac{1}{21\sqrt{\ln n}}}.
    \end{align*}
    We shift weight from one edge to another at most $2n\cdot\frac{\eps n}{2}=\eps n^2$ times.
    Using $h(\textbf{y}) = \sum_{e\in E(B')} \lambda\textbf{x}'_e\log\frac{1}{\lambda\textbf{x}'_e} = \lambda\sum_{e\in E(B')} \textbf{x}'_e\log\frac{1}{\textbf{x}'_e} - \sum_{e\in E(B')}\lambda\textbf{x}'_e\log\lambda = \lambda\sum_{e\in E(B')} \textbf{x}'_e\log\frac{1}{\textbf{x}'_e} - n'\log\lambda$, this yields
    \begin{align}\label{equation: h(z)}
        h(\textbf{z})
        &\geq h(\textbf{y}) - n^{\frac14-\frac{1}{21\sqrt{\ln n}}}
        = \lambda \sum_{e\in E(B')}\textbf{x}'_e\log\frac{1}{\textbf{x}'_e} - n' \log\lambda - n^{\frac14-\frac{1}{21\sqrt{\ln n}}}.
    \end{align}
    Note that by~\eqref{equation: bounds for lambda}, $\lambda \leq \frac{n-m}{n-2m} = \frac{n}{n'}(1+\frac{m^2}{n^2-2mn}) \leq \frac{n}{n'}(1+\frac{2m^2}{n'^2})$ and thus $\log\lambda \leq \log\frac{n}{n'} + \log(1+\frac{2m^2}{n'^2}) \leq \log\frac{n}{n'} + \frac{4m^2}{n'^2}$. Combining this with~\eqref{equation: h(z)} yields
    \begin{align}\label{equation: h(z) the second}
        h(\textbf{z})
        &\geq \lambda \sum_{e\in E(B')}\textbf{x}'_e\log\frac{1}{\textbf{x}'_e} - n' \log\frac{n}{n'} - n^{\frac14-\frac{1}{22\sqrt{\ln n}}}.
    \end{align}
    We bound $\sum_{e\in E(B')}\textbf{x}'_e\log\frac{1}{\textbf{x}'_e}$ in terms of~$h(\textbf{x})$.
    Recall that $E(M,G):=\{vw\in E(G)\mid \text{$v\in M$ or $w\in M$}\}$. Using (ii) we obtain
    \begin{align*}
        \sum_{v\in V(G)}\sum_{w\in M\cap N_G^+(v)} \textbf{x}_{vw}\log\frac{1}{\textbf{x}_{vw}} = \frac{m+1}{n}\sum_{v\in V(G)}h_{\textbf{x}}^+(v) \pm n^{\frac14-\frac{1}{18\sqrt{\ln n}}}
        = \frac{m}{n}h(\textbf{x}) \pm n^{\frac14-\frac{1}{19\sqrt{\ln n}}}
    \end{align*}
    and the analogous inequality for~$N_G^-(v)$ and~$h_{\textbf{x}}^-(v)$. Thus,
    \begin{align*}
        h_{\textbf{x}}(E(M,G)) &\leq \sum_{v\in V(G)} \left(\sum_{w\in M\cap N_G^+(v)} \textbf{x}_{vw}\log\frac{1}{\textbf{x}_{vw}} + \sum_{w\in M\cap N_G^-(v)} \textbf{x}_{wv}\log\frac{1}{\textbf{x}_{wv}}\right) \\
        &= 2\frac{m}{n}h(\textbf{x}) \pm n^{\frac14-\frac{1}{20\sqrt{\ln n}}}.
    \end{align*}
    This yields
    \begin{align}\label{equation: sum e in B'}
        \sum_{e\in E(B')} \textbf{x}'_e\log\frac{1}{\textbf{x}'_e}
        &\geq h(\textbf{x}) - h_{\textbf{x}}(E(M,G))
        \geq \frac{n'}{n}h(\textbf{x}) - \frac{m}{n}h(\textbf{x}) - n^{\frac14-\frac{1}{20\sqrt{\ln n}}}.
    \end{align}
    By combining~\eqref{equation: h(z) the second} and~\eqref{equation: sum e in B'} and using $\lambda\geq 1 + (m-\frac{2m^2}{n})/(n-2m+\frac{2m^2}{n})$ (see~\eqref{equation: bounds for lambda}) and $h(\textbf{x}) \leq n\log n$, we obtain
    \begin{align*}
        h(\textbf{z})
        &\geq \left(1 + \frac{m-\frac{2m^2}{n}}{n-2m+\frac{2m^2}{n}}\right) \left(\frac{n'}{n} h(\textbf{x}) - \frac{m}{n}h(\textbf{x})\right) - n'\log\frac{n}{n'} - n^{\frac14-\frac{1}{23\sqrt{\ln n}}} \\
        &\geq \frac{n'}{n}h(\textbf{x}) - \frac{4m^2}{n^2}h(\textbf{x}) - n'\log\frac{n}{n'} - n^{\frac14-\frac{1}{23\sqrt{\ln n}}} \\
        &\geq \frac{n'}{n}h(\textbf{x}) -n'\log\frac{n}{n'} - n^{\frac14-\frac{1}{24\sqrt{\ln n}}},
    \end{align*}
    as desired.
    
    Next we show that~$\textbf{z}$ is $b'$-normal for some suitable~$b'$. We need to choose~$b'$ such that $\textbf{z}_e \leq \frac{b'}{n'}$ for each $e\in E(B')$. 
    Let $\alpha := n^{-\frac74-\frac{1}{20\sqrt{\ln n}}}$.
    Note that trivially $\frac{1}{bn} \leq \textbf{x}_{uv},\textbf{x}_{wu}\leq\frac{b}{n}$ for any $v,w\in V(G)$ such that $uv,vw\in E(G)$.
    Thus, $\textbf{z}_e\leq\frac{\lambda b}{n}+\alpha$ by~\eqref{equation: z_e}, and we need to choose~$b'$ to satisfy $b'\geq \frac{\lambda bn'}{n} + \alpha n'$.
    Let $\zeta := n^{-\frac34-\frac{1}{22\sqrt{\ln n}}}$ and $b':=(1+\zeta)b$.
    Note that $\lambda+\alpha n \leq 1 + \frac{m}{n-2m} + n^{-\frac34-\frac{1}{20\sqrt{\ln n}}} \leq 1+\frac{m + \zeta n}{n-m}$.%
    \COMMENT{
        we have $
            m^2 + (n-m)(n-2m)n^{-\frac34-\frac{1}{20\sqrt{\ln n}}}
            \leq m^2 + n^{\frac54-\frac{1}{20\sqrt{\ln n}}}
            \leq 2n^{\frac54-\frac{1}{20\sqrt{\ln n}}}
            \leq n^{\frac54-\frac{1}{21\sqrt{\ln n}}}
            \leq \zeta n(n-2m)
        $
        thus
        $
            m(n-m) + (n-m)(n-2m)n^{-\frac34-\frac{1}{20\sqrt{\ln n}}} \leq (m+\zeta n)(n-2m)
        $
        which yields the result by dividing through $(n-m)(n-2m)$.
    }
    Then \begin{align*}
        \frac{\lambda b n' + \alpha nn'}{n}
        \leq \frac{(\lambda+\alpha n)n'}{n} b
        \leq \frac{(1+\frac{m+\zeta n}{n'})n'}{n}b
        = (1+\zeta)b
        = b'.
    \end{align*}
    
    Note that~$b'$ must also satisfy $\textbf{z}_e \geq \frac{1}{b'n'}$ for each $e\in E(B')$. As $\textbf{z}_e \geq \frac{\lambda}{bn} - \alpha$ by~\eqref{equation: z_e}, we need to choose~$b'$ such that $b'\geq\frac{1}{(\frac{\lambda}{bn}-\alpha)n'}$. Note that $(\lambda-\alpha bn)(n-m) \geq n-n^{\frac14-\frac{1}{21\sqrt{\ln n}}}$.%
    \COMMENT{
        As $\lambda\geq 1+\frac{m-\frac{2m^2}{n}}{n-2m+\frac{2m^2}{n}} \geq 1 + \frac{m-\frac{2m^2}{n}}{n}$, we have
        \begin{align*}
            \lambda(n-m)
            &\geq \left(1+\frac{m-\frac{2m^2}{n}}{n}\right)(n-m)
            \geq n - \frac{3m^2}{n}
        \end{align*}
        and thus $(\lambda-\alpha bn)n' \geq n - \frac{3m^2}{n} - bn^{\frac14-\frac{1}{20\sqrt{\ln n}}} \geq n - n^{\frac14-\frac{1}{21\sqrt{\ln n}}}$.
    }
    Hence%
    \COMMENT{
        as $\zeta \geq (1+\zeta)n^{-\frac34-\frac{1}{21\sqrt{\ln n}}}$, we have
            $
                (1-n^{-\frac34-\frac{1}{21\sqrt{\ln n}}})(1+\zeta)
                = 1-n^{-\frac34-\frac{1}{21\sqrt{\ln n}}}+\zeta -\zeta n^{-\frac34-\frac{1}{21\sqrt{\ln n}}}
                \geq 1.
            $
    }
    \begin{align*}
        \frac{1}{(\frac{\lambda}{bn}-\alpha)n'}
        = \frac{n}{(\lambda - \alpha bn)n'}b
        \leq \frac{1}{1-n^{-\frac34-\frac{1}{21\sqrt{\ln n}}}}b
        \leq (1+\zeta) b
        = b'.
    \end{align*}
    
    Altogether we obtain that~$\textbf{z}$ is $b'$-normal with \mbox{$b'\leq(1+n^{-\frac34-\frac{1}{22\sqrt{\ln n}}})b$}.
\end{proof}

The next lemma shows that if $A\subseteq V(G)$ is not too large, the graph~$G-A$ has a $b$-normal perfect fractional matching whose entropy is only slightly less than~$h(G)$.
\begin{lemma}\label{lemma: matching G-A and G}
    Suppose $\frac1n\ll\frac{1}{b}\ll\eps$. Let~$G$ be an $(n,\eps)$-digraph and suppose $A\subseteq V(G)$ satisfies $|A|\leq \frac{n}{\log^2 n}$.
    Then there is a $b$-normal perfect fractional matching~$\textbf{z}$ of~$G-A$ that satisfies $h(\textbf{z})\geq h(G)-\eps n$.
\end{lemma}
\begin{proof}
    Note that $h(G)\geq n\log\delta^0(G) \geq n\log\frac{n}{2}+\eps^2n$ (see~\cite[Theorem~3.1]{CK:09}).
    According to Theorem~\ref{theorem: almost maximum matching}, the digraph~$G$ has a $\frac{b}{2}$-normal perfect fractional matching~$\textbf{x}$ that satisfies $h(\textbf{x})\geq h(G)-\eps^2 n$.
    We show that we can modify~$\textbf{x}$ to obtain a $b$-normal perfect fractional matching~$\textbf{z}$ that satisfies $h(\textbf{z})\geq h(\textbf{x})-3|A|\log n$, which implies the desired result. To this end, we follow a similar (but simpler) strategy to Lemma~\ref{lemma: new matching}. 
    
    Define $a:=|A|$ and $n':=|V(G)\setminus A|$. Let $B':=B_{G-A}$ be the balanced bipartite graph on~$2n'$ vertices corresponding to~$G-A$.
    Let $\lambda:=\frac{n'}{\sum_{e\in E(B')}\textbf{x}_e}$ and $\textbf{y}_e:=\lambda\textbf{x}_e$ for $e\in E(B')$. Then $\sum_{e\in E(B')} \textbf{y}_e=n'$.
    Let~$E(A,G):=\{vw\in E(G)\mid \text{$v\in A$ or $w\in A$}\}$.
    Note that $a\leq \sum_{e\in E(A,G)}\textbf{x}_e \leq 2a$ and hence $n-2a \leq \sum_{e\in E(B')} \textbf{x}_e\leq n-a$.
    Thus, $1\leq \lambda\leq\frac{n-a}{n-2a} = 1+\frac{a}{n-2a}$.

    Analogously to Lemma~\ref{lemma: new matching}, we redistribute weights on edges of~$B'$ to obtain a perfect fractional matching~\textbf{z} by shifting weights on paths of length~2 between vertices~$v\in V(B')$ that satisfy $\sum_{v'\in N_{B'}(v)}\textbf{y}_{vv'}>1$ and vertices $w\in V(B')$ that satisfy $\sum_{w'\in N_{B'}(w)}\textbf{y}_{ww'}<1$.
    As in Lemma~\ref{lemma: new matching}, we can always choose~$\frac{\eps n}{2}$ paths of length~2 between two such vertices to redistribute weights, and thus, $\textbf{z}$ satisfies $\textbf{z}_{e}=\textbf{y}_{e} \pm \frac{4x}{\eps n}$ for each $e\in E(B')$, where $x:=\max_{v\in V(B')}{|\sum_{v'\in N_{B'}(v)}\textbf{y}_{vv'}-1|}$.
    
    We bound~$x$ from above. Let $v\in V(G)\setminus A$ and $v^+, v^-$ be its corresponding vertices in~$V(B')$.
    \textcolor{black}{As~$\textbf{x}$ is $\frac{b}{2}$-normal, we have
    \begin{align*}
        1=\sum_{w\in (V(G)\setminus A)\cap N_{G}^+(v)}\textbf{x}_{vw} + \sum_{w\in A\cap N_G^+(v)} \textbf{x}_{vw} \leq \sum_{v'\in N_{B'}(v^+)}\textbf{x}_{v^+v'} + \frac{ab}{2n},
    \end{align*}
    and the analogous inequality holds for~$v^-$.}
    Thus, $1-\frac{ab}{2n} \leq \sum_{v'\in N_{B'}(v)}\textbf{x}_{vv'} \leq 1$. Using $1\leq\lambda\leq 1+\frac{a}{n-2a}$ yields $1-\frac{ab}{2n}\leq \sum_{v'\in N_{B'}(v)}\textbf{y}_{vv'}\leq 1+\frac{a}{n-2a} \leq 1+\frac{ab}{2n}$.
    Thus, for each $v\in V(B')$, we need to redistribute at most weight $\frac{ab}{2n}$ to achieve $\sum_{v'\in N_{B'}(v)}\textbf{y}_{vv'}=1$, and in this process every edge weight is changed by at most~$\frac{2ab}{\eps n^2}$. Hence,
    \begin{equation*}
        \textbf{z}_e = \textbf{y}_e \pm \frac{2ab}{\eps n^2}
    \end{equation*}
    for each $e\in E(B')$. As $\frac1n\ll\frac1b\ll\eps$, we conclude that $\textbf{z}$ is a $b$-normal perfect fractional matching.

    We now bound the change of entropy in this process from above. Assume that~$e,e'\in E(B')$ are used to redistribute weight, and assume that weight~$\alpha\leq \frac{2ab}{\eps n^2}$ is subtracted from~$\textbf{y}_e$ and added to~$\textbf{y}_{e'}$.
    Analogously to Lemma~\ref{lemma: new matching} we obtain
    \begin{align*}
        h_{\textbf{z}}(\{e, e'\})
        &\geq \textbf{y}_e\log\frac{1}{\textbf{y}_e} + \textbf{y}_{e'}\log\frac{1}{\textbf{y}_{e'}}
        -5\alpha\log b
        \geq \textbf{y}_e\log\frac{1}{\textbf{y}_e} + \textbf{y}_{e'}\log\frac{1}{\textbf{y}_{e'}}
        - \frac{ab^2}{\eps n^2}.
    \end{align*}
    We shift at most~$2n\cdot\frac{\eps n}{4}\leq\eps n^2$ times weight from one edge to another.
    This yields
    \begin{equation*}
        h(\textbf{z}) \geq h(\textbf{y}) - ab^2
        = \sum_{e\in E(B')} \lambda\textbf{x}_e\log\frac{1}{\lambda\textbf{x}_e} - ab^2
        \geq \sum_{e\in E(B')}\textbf{x}_e\log\frac{1}{\textbf{x}_e} - ab^2,
    \end{equation*}
    as $\lambda\geq1$ and $x'\log\frac{1}{x'}\geq x\log\frac{1}{x}$ if $x'\geq x$ as long as $x,x'\leq e^{-1}$.
    Note that $h_{\textbf{x}}^+(v) = \sum_{w\in N_G^+(v)}\textbf{x}_{vw}\log\frac{1}{\textbf{x}_{vw}} \leq \log n$ for any $v\in V(G)$ by Jensen's inequality, and the same holds for~$h_{\textbf{x}}^-(v)$.
    Thus,
    \begin{align*}
        h_{\textbf{x}}(E(A,G)) \leq \sum_{v\in A} h_\textbf{x}^+(v) + h_\textbf{x}^-(v)
        \leq 2a\log n.
    \end{align*}
    This yields $h(\textbf{z}) \geq h(\textbf{x}) - h_{\textbf{x}}(E(A,G)) - ab^2 \geq h(\textbf{x}) - 3a\log n$.
    Altogether, $h(\textbf{z}) \geq h(G) - 3a\log n - \eps^2 n \geq h(G) - \eps n$.
\end{proof}

\section{Conditions for a random tree being well-behaved}\label{sec:Conditions for a random tree being well-behaved}

In this section, we define what it means for a (random) tree in a digraph~$G$ on~$n$ vertices to be \emph{self-avoiding} and \mbox{\emph{$(\cS,a,c)$-expected}}.
We prove that a tree of~$T$ with $n^\frac14+1\leq |T|\leq 3\Delta n^\frac14$ satisfies both properties with high probability.
This yields the foundation for applying Lemma~\ref{lemma: entropy} to show
that~$G$ contains not only a large number of copies of a tree of this order,
but also a large number of \emph{well-behaved} copies (such a copy is well-behaved if it is $(\cS,n^{\frac14-\frac{1}{17\sqrt{\ln n}}},n^{-\frac34-\frac{1}{18\sqrt{\ln n}}})$-expected for a collection of sets~$\cS$ that is specified in Section~\ref{sec:Counting trees}).

\begin{defin}\label{defin: (S,a,c)-expected}
Let~$G$ be a digraph and~$\textbf{x}$ a perfect fractional matching of~$G$.
Let~$T=(r_0,\ldots,r_m)$ be an oriented tree and $R=(R_0,\ldots,R_m)$ a random tree according to~$T$ in~$G$.
We call~$R$ \emph{self-avoiding} if $R_i\neq R_j$ for distinct $i,j\in[m]_0$.
Let~$\cS$ be a collection of subsets in~$V(G)$ and let $M\subseteq V(G)$.
We say that~$M$ is \mbox{\emph{$(\cS,a)$-expected}} if $||M\cap S| - \frac{|M|}{n}|S|| < a$ for each $S\in\cS$.
We say that~$M$ is \emph{$c$-expected} if $|\sum_{w\in N^+(v)\cap M} \textbf{x}_{vw} - \frac{|M|}{n}| < c$ and $|\sum_{w\in N^+(v)\cap M}\textbf{x}_{vw}\log\frac{1}{\textbf{x}_{vw}} - \frac{|M|}{n}h_{\textbf{x}}^+(v)| < c$ for each vertex $v\in V(G)$, and if the analogous results hold for~$N^-(v)$ and~$h_{\textbf{x}}^-(v)$.
Finally, we say that~$M$ is $(\cS,a,c)$-expected if it is both $(\cS,a)$-expected and $c$-expected.
We transfer these notions to a subgraph~$Q$ in~$G$ by setting $M:=V(Q)$.
\end{defin}

The next lemma shows that a random tree is self-avoiding with high probability. It is essentially already proved in~\cite{CK:09} and follows directly from the upper bound on~$\textbf{x}$ and an application of the union bound.
\begin{lemma}[{\cite[Lemma~5.1]{CK:09}}]\label{lemma: self-avoiding}
	Let~$G$ be a digraph on~$n$ vertices and~$\textbf{x}$ a perfect fractional matching of~$G$.
	Let~$T$ be an oriented tree on~$m+1$ vertices and~$R$ a random tree in~$G$ according to~$T$.
	Suppose every edge $e\in E(G)$ satisfies $\textbf{x}_e \leq \frac{b}{n}$ for some $b\geq1$.
	Then%
    \COMMENT{
    We have $\pr[R_i\in \{R_0,\ldots,R_{i-1}\}]\leq i\frac{b}{n}$.
    Thus,
    \begin{equation*}
        \pr[(R_0,\ldots,R_m)\text{ not self-avoiding}]
        = \sum_{i=1}^m \pr[R_i\in\{R_0,\ldots,R_{i-1}\}]
        \leq \sum_{i=1}^m i\frac{b}{n}
        \leq m^2\frac{b}{n}.
    \end{equation*}
    }
	\begin{equation*}
		\pr[R\text{ is self-avoiding}] \geq 1 - m^2\frac{b}{n}.
	\end{equation*}
\end{lemma}
\begin{proof}
    We have $\pr[R_i\in \{R_0,\ldots,R_{i-1}\}]\leq i\frac{b}{n}$.
    Thus,
    \begin{equation*}
        \pr[(R_0,\ldots,R_m)\text{ not self-avoiding}]
        = \sum_{i=1}^m \pr[R_i\in\{R_0,\ldots,R_{i-1}\}]
        \leq \sum_{i=1}^m i\frac{b}{n}
        \leq m^2\frac{b}{n},
    \end{equation*}
    which proves the claim.
\end{proof}

Thew following lemma yields a lower bound on the speed of convergence of a Markov chain.
\begin{lemma}[{\cite[Lemma~3.2]{JK:21}}]\label{lemma: rapidlymixing}
	Let $(Z_t)_{t\in\bN_0}$ be a Markov chain with state space $\{s_1,\ldots,s_n\}$ and
	transition matrices $P^{(t)}=(p^{(t)}_{ij})_{i,j\in[n]}$.
	Suppose there is a $\sigma=(\sigma_i)_{i\in[n]}$ with $\sigma_i>0$ for all $i\in\bN$ such that $\sum_{i\in[n]}\sigma_i=1$ and $\sigma P^{(t)} = \sigma$ for all $t\in\bN_0$.
	Let
	\[
	    \alpha:=\inf_{t\in\bN_0,i,j,k\in[n]}\frac{p^{(t)}_{ij}}{\sigma_k}, 
	    \hspace{5mm}
	    \beta:=\sup_{t\in\bN_0,i,j,k\in[n]}\frac{p^{(t)}_{ij}}{\sigma_k}.
	\]
	Assume $\alpha>0$. Then
	\begin{equation*}
		\pr[Z_t = s_i] = (1\pm (1-\textstyle{\frac\alpha2)^t})\sigma_i
	\end{equation*}
	holds for $t\geq 2+2\alpha^{-1}\log\beta$.
\end{lemma}

Now we utilise Lemma~\ref{lemma: rapidlymixing} for our setting.
\begin{cor}\label{cor: probability}
	Suppose $\frac1n\ll\eps$. Let~$G$ be an \mbox{$(n,\eps)$-digraph}, let~$\textbf{x}$ be a perfect fractional matching of~$G$ and let $(Z_t)_{t\in\bN_0}$ be a random walk on~$G$ induced by~$\textbf{x}$ following the pattern of an oriented path~$L=(y_t)_{t\in\bN_0}$.
	Suppose~$\textbf{x}$ is \mbox{$b$-normal} for some $b\geq1$.
	Then, for any vertex $v\in V(G)$, we have
	\begin{equation*}
		\pr[Z_t=v] = (1\pm e^{-\frac{\eps}{2b^2}t})\frac1n 
	\end{equation*}
	for $t\geq 5+4b^2\eps^{-1}\log b$.
\end{cor}
\begin{proof}
    Denote $p_{vw}:=\textbf{x}_{vw}$ if $vw\in E(G)$ and $p_{vw}:=0$ otherwise. Let $P:=(p_{vw})_{v,w\in V(G)}$.
    Note that~$P$ is the transition matrix of~$Z$ at times~$t$ where $y_ty_{t+1}\in E(L)$ and $P^T$ is the transition matrix of~$Z$ at times~$t$ where $y_{t+1}y_t\in E(L)$.
    As $\delta^0(G)\geq(\frac{1}{2}+\eps)n$,
	we have $|A\cap B| \geq 2\eps n$ for all $v,w\in G$, where $A\in\{N^+(v),N^-(v)\}$ and $B\in\{N^+(w),N^-(w)\}$.
	Suppose $y_iy_{i+1},y_{i+1}y_{i+2}\in E(L)$. As $p_{vw}\geq\frac{1}{bn}$ for $vw\in E(G)$,
	we obtain
	\begin{equation*}
		q^{(i)}_{vw}:=\pr[Z_{i+2}=w \mid Z_i = v] \geq |N^+(v)\cap N^-(w)|\min_{u\in N^+(v)\cap N^-(w)} p_{vu}p_{uw} \geq \frac{2\eps}{b^2n}
	\end{equation*}
	for $v,w\in V(G)$.
	The analogous result holds if the path~$L$ is oriented differently.
	Note that, as~$P$ and~$P^T$ are doubly stochastic, so are $P^2,(P^T)^2,PP^T,P^TP$.
    Thus, $\sigma:=(\sigma_v)_{v\in V(G)} := \frac1n(1,\ldots,1)$ satisfies $\sigma_v>0$ for all $v\in V(G)$, $\sum_{v\in V(G)}\sigma_v=1$ and $\sigma Q =\sigma$ for each $Q\in\{P^2,(P^T)^2,PP^T,P^TP\}$.
    Let $Z'_i:=Z_{2i}$. Then $(Z'_i)_{i\in\bN_0}$ defines a Markov chain with state space~$V(G)$ and potential transition matrices $P^2,(P^T)^2,PP^T,P^TP$.
	Let $\alpha:=\min_{i,v,w,u}\frac{q^{(i)}_{vw}}{\sigma_u}\geq\frac{2\eps}{b^2}$ and $\beta:=\max_{i,v,w,u}\frac{q^{(i)}_{vw}}{\sigma_u}=n\cdot\max_{i,v,w}q^{(i)}_{vw} \leq n\frac{b}{n}n\frac{b}{n} = b^2$.
	Using the inequality $(1-x)^t \leq e^{-xt}$,
	Lemma~\ref{lemma: rapidlymixing} yields
	\begin{equation*}
		\pr[Z_t=v] = \pr[Z'_{\frac{t}{2}}=v] = (1\pm(1-\textstyle{\frac{\alpha}{2}})^\frac{t}{2})\sigma_v = (1\pm(1-\textstyle{\frac{\eps}{b^2}})^\frac{t}{2})\frac{1}{n}
        = (1\pm e^{-\frac{\eps}{2b^2}t})\frac{1}{n}
	\end{equation*}
	for any vertex $v\in V(G)$, for any even~$t$ that satisfies $t\geq 4+4b^2\eps^{-1}\log b$.
    Similarly we conclude $\pr[Z_{i+t}=v\mid Z_i=w] = (1\pm e^{-\frac{\eps}{2b^2}t})\frac{1}{n}$ for any $v,w\in V(G)$ and even $t\geq 4+4b^2\eps^{-1}\log b$.
    Thus, for any odd~$t$ that satisfies $t\geq 5+4b^2\eps^{-1}\log b$, we obtain
    \textcolor{black}{\begin{align*}
        \pr[Z_t=v]
        = \sum_{w\in V(G)}\pr[Z_1=w]\pr[Z_t=v\mid Z_1=w]
        = (1\pm e^{-\frac{\eps}{2b^2}t})\frac1n.
    \end{align*}}
    This yields the desired result.
\end{proof}

The following lemma uses Corollary~\ref{cor: probability} to show that a random tree is very likely to behave as one would expect with regard to a subset $S\subseteq V(G)$. To be more precise, given a vertex weighting of~$G$, with high probability the sum of weights of the vertices in~$S$ that are hit by the random tree does not deviate too much from its expected value.
\begin{lemma}\label{lemma: neighbours tree}
	Suppose $\frac1n\ll\gamma\ll\frac1b\ll\eps$. Let~$G$ be an \mbox{$(n,\eps)$-digraph}, let $S\subseteq V(G)$, let~$\textbf{x}$ be a \mbox{$b$-normal} perfect fractional matching of~$G$, let $c\in\bR$ and let $\textbf{y}\colon V(G)\to\bR_{\geq0}$ be a vertex weighting with $\textbf{y}_v\leq c$ for all $v\in V(G)$.
	Let $\Delta:= e^{\gamma\sqrt{\ln n}}$
	and let $T=(r_0,\ldots,r_m)$ be an oriented tree with $n^\frac14+1\leq |T|\leq 3\Delta n^\frac14$ and $\Delta(T)\leq\Delta$.
	Let $R=(R_0,\ldots,R_m)$ be a random tree in~$G$ according to~$T$.
	Then
	\begin{equation*}
		\pr\left[\left|\sum_{j=0}^m \textbf{y}_{R_j}\1_{R_j\in S} - \frac{m+1}{n}\sum_{v\in S}\textbf{y}_{v}\right| \geq cn^{\frac14-\frac{1}{17\sqrt{\ln n}}} \right] \leq \frac{1}{n^3}.
	\end{equation*}
\end{lemma}
\begin{proof}
	Our strategy for the proof is to consider a partition of~$T$ into rooted subtrees. For each of those subtrees we ignore the vertices close to the root and then show that the behaviour for any pair of vertices in different subtrees is essentially independent.
	To this end, we adapt ideas of Lemma~6.6 in~\cite{JKKO:19}.
	
	Let $\ell=n^\frac14$.
	As in Lemma~\ref{lemma: tree partition}, we partition~$T$ into vertex-disjoint rooted subtrees $(T_1,t_1),\ldots,(T_k,t_{k})$
	such that $T_i\subseteq T(t_i)$ for every $i\in[k]$, and each subtree satisfies $\ell^{1-\zeta}\leq |T_i|\leq 2\Delta\ell^{1-\zeta}$,
	where $\zeta:=\frac{\gamma}{{\ln \Delta}}=\frac{1}{\sqrt{\ln n}}$.
	Note that $\frac{\ell^\zeta}{2\Delta}\leq k\leq 3\Delta\ell^\zeta$.
	Let~$T_i'$ denote the subforest of $T_i$ induced by vertices of~$T_i$ of depth at least~$\frac12\log_\Delta\ell$ in~$T_i'$.
	Then $|T_i-T_i'|\leq\Delta^{\frac12\log_\Delta\ell}=\ell^\frac12$.
	Let $X_j:=\textbf{y}_{R_j}\1_{R_j \in S}$ for $j\in[m]_0$. Let $J_i:=\{j\in[m]_0 \mid r_j\in V(T_i')\}$ and let $Y_i:=\sum_{j\in J_i}\textbf{y}_{R_j}\1_{R_j\in S}$ for $i\in[k]$.
    Note that
	\begin{equation}\label{equation: sum X_j}
		\sum_{j=0}^m X_j \leq c\sum_{i=1}^k |T_i-T_i'| + \sum_{i=1}^k Y_i \leq 3c\Delta\ell^{\frac12+\zeta} + \sum_{i=1}^k Y_i.
	\end{equation}
	Consider the exposure martingale $Z_i:=\ex[\sum_{i'=1}^k Y_{i'} \mid Y_1,\ldots,Y_i]$.
    We show that $|Z_i-Z_{i-1}|$ is not too large for any $i\in[k]$. Then we apply Azuma's inequality to show that $Z_0=\ex[\sum_{i=1}^k Y_{i}]$ is close to $Z_k=\sum_{i=1}^k Y_{i}$ with high probability. As~$\sum_{i=1}^k Y_{i}$ is close to $\sum_{j=0}^m X_j$ and~$\ex[\sum_{i=1}^k Y_{i}]$ is close to $\frac{m+1}{n}\sum_{v\in S}\textbf{y}_v$, this yields the desired result.

    Let $i\in[k]$ and $y_1,\ldots,y_i,y_i'\in\bN_0$ be such that $y_{i'}\leq c|T_{i'}'|$ for each ${i'}\in[i]$ and $y_i'\leq c|T_i'|$.
	Let~$\cE$ be the event that $Y_1=y_1,\ldots,Y_i=y_i$ occurs and let~$\cE'$ be the event that $Y_1=y_1,\ldots,Y_{i-1}=y_{i-1},Y_i=y_i'$ occurs.
    Let $i'>i$ and $j\in J_{i'}$.
    Note that the length of the path from~$t_{i'}$ to any vertex in~$T_{i'}'$ is at least $\frac12\log_\Delta\ell$.
    \textcolor{black}{By the definition of a random tree, this path, when embedded into~$G$, can be regarded as a random walk.}
    Thus for any $v\in V(G)$ and $j\in J_{i'}$, Corollary~\ref{cor: probability} yields $\pr[R_j=v\mid\cE] = (1\pm e^{-\frac{\eps}{2b^2}\frac12\log_\Delta \ell})\frac1n = (1\pm\ell^{-\zeta})\frac1n$, where the last inequality follows from $\gamma\ll\frac{1}{b}\ll\eps$.%
    \COMMENT{As $\gamma\ll\frac{1}{b}\ll\eps$, we have $\frac{\eps}{4b^2}\geq\gamma$. Hence $e^{-\frac{\eps}{4b^2}\log_\Delta\ell} \leq e^{-\frac{\gamma}{\ln\Delta}\ln\ell} = e^{-\zeta\ln\ell} = \ell^{-\zeta}$.}%
    \COMMENT{
        Note that this calculation is the reason we have to choose $\zeta\leq\frac{\gamma}{\ln\Delta}$.
    }
    This implies
	\begin{align} \label{equation: exp_of_Y_j_2}
		\ex[Y_{i'} \mid \cE]
        &= \sum_{j\in J_{i'}} \sum_{v\in S} \textbf{y}_{v} \ex[\1_{R_j=v} \mid \cE]
        = \sum_{j\in J_{i'}}\sum_{v\in S} \textbf{y}_{v} \pr[R_j=v \mid \cE]
		= (1\pm\ell^{-\zeta}) \frac{|T_{i'}'|}{n}\sum_{v\in S}\textbf{y}_v.
	\end{align}
    Thus,
            \COMMENT{
            As $\textbf{y}_v \leq c$ and $|S|\leq n$ as well as $\sum_{i'=i+1}^k|T_{i'}'|\leq |T| \leq 3\Delta\ell$, we have
            \begin{align*}
		      \frac{2}{n}\ell^{-\zeta}\sum_{v\in S}\textbf{y}_v \sum_{i'=i+1}^k |T_{i'}'|
                \leq 2c\ell^{-\zeta}\sum_{i'=i+1}^k|T_{i'}'|
		      \leq 6c\Delta\ell^{1-\zeta}.
            \end{align*}
            }
	\begin{align*}
	   \left|\ex\left[\sum_{i'=1}^k Y_{i'}\mid\cE\right]-\ex\left[\sum_{i'=1}^k Y_{i'}\mid\cE'\right]\right|
		&\leq |y_i-y_i'| + \sum_{i'=i+1}^k |\ex[Y_{i'}\mid\cE]-\ex[Y_{i'}\mid\cE']| \\
		&\leq c|T_i'| + \sum_{i'=i+1}^k \left|(1\pm\ell^{-\zeta})\frac{|T_{i'}'|}{n}\sum_{v\in S}\textbf{y}_v - (1\pm\ell^{-\zeta})\frac{|T_{i'}'|}{n}\sum_{v\in S}\textbf{y}_v\right| \\
		&\leq 2c\Delta\ell^{1-\zeta} + \frac{2}{n}\ell^{-\zeta}\sum_{v\in S}\textbf{y}_v \sum_{i'=i+1}^k |T_{i'}'|
		\leq 8c\Delta\ell^{1-\zeta}.
	\end{align*}
	Hence, $|Z_i-Z_{i-1}|\leq 8c\Delta\ell^{1-\zeta} = 8c\Delta n^{\frac14-\frac\zeta4}$ for all $i\in[k]$.
	Lemma~\ref{lemma: Azuma}%
        \COMMENT{
            Note that this calculation is the reason we need $\Delta\leq e^{\gamma\sqrt{\ln n}}$. This is due to the fact that the last inequality in the application of Azuma's inequality is true essentially if $\zeta\geq 49\frac{\ln\Delta}{\ln n}$. As $\zeta=\frac{\gamma}{\ln\Delta}$, this is equivalent to $\ln^2\Delta \leq \frac{\gamma}{49}\ln n$, which means $\Delta \leq e^{\frac{\sqrt{\gamma}}{7}\sqrt{\ln n}}$ (and we simply replace $\gamma$ by $\sqrt{\gamma}$).
        }
        yields%
        \COMMENT{
            As $k\leq 3\Delta\ell^\zeta = 3\Delta n^{\frac{\zeta}{4}}$, we have $-\frac{1}{128k}n^{\frac{\zeta}{2}-\frac{\zeta}{8}} \leq -\frac{1}{10^3\Delta}n^{\frac{\zeta}{2}-\frac{\zeta}{4}-\frac{\zeta}{8}}$ and thus the second inequality follows
        }
	\begin{align*}
		\pr\left[\left|Z_k-Z_0\right|\geq c\Delta n^{\frac14-\frac{\zeta}{16}}\right]
		&\leq 2\exp\left(-\frac{c^2\Delta^2n^{\frac12-\frac\zeta8}}{2k(8c\Delta n^{\frac14-\frac\zeta4})^2}\right)
		= 2\exp\left(-\frac{1}{128k}n^{\frac{3}{8}\zeta}\right) \\
		&\leq 2\exp\left(-\frac{1}{384\Delta} n^{\frac\zeta8}\right)
		\leq \frac{1}{n^3}.
	\end{align*}
	We have $Z_k=\sum_{i=1}^k Y_{i}$ and $Z_0 = \ex[\sum_{i=1}^k Y_{i}] = \sum_{i=1}^k \ex Y_{i} = (1\pm\ell^{-\zeta})\frac{1}{n}\sum_{v\in S}\textbf{y}_v\sum_{i=1}^k |T_{i}'|$, which follows as in~\eqref{equation: exp_of_Y_j_2}.
	As $m+1-3\Delta\ell^{\frac12+\zeta}\leq\sum_{i=1}^k|T_{i}'|\leq m+1$, we obtain%
        \COMMENT{
            As $\sum_{v\in S}\textbf{y}_v\leq c|S|\leq cn$, we have 
            $Z_0 \leq (1+\ell^{-\zeta})\frac1n(m+1)\sum_{v\in S}\textbf{y}_v 
            = \frac{m+1}{n}\sum_{v\in S}\textbf{y}_v + \ell^{-\zeta}\frac{m+1}{n}\sum_{v\in S}\textbf{y}_v 
            \leq \frac{m+1}{n}\sum_{v\in S}\textbf{y}_v + \ell^{-\zeta}3\Delta\ell c
            \leq \frac{m+1}{n}\sum_{v\in S}\textbf{y}_v + 3c\Delta\ell^{1-\zeta}$ 
            and
            $Z_0\geq (1-\ell^{-\zeta})\frac{1}{n}(m+1-3\Delta\ell^{\frac12+\zeta})\sum_{v\in S}\textbf{y}_v 
            = \frac{m+1}{n}\sum_{v\in S}\textbf{y}_v - \frac{3\Delta\ell^{\frac12+\zeta}}{n}\sum_{v\in S}\textbf{y}_v - \ell^{-\zeta}\frac{m+1-3\Delta\ell^{\frac12+\zeta}}{n}\sum_{v\in S}\textbf{y}_v
            \geq \frac{m+1}{n}\sum_{v\in S}\textbf{y}_v - 3c\Delta\ell^{\frac12+\zeta} - c\ell^{-\zeta}(m+1-3\Delta\ell^{\frac12+\zeta})
            \geq \frac{m+1}{n}\sum_{v\in S}\textbf{y}_v - 3c\Delta\ell^{\frac12+\zeta} - 3c\Delta\ell^{1-\zeta}
            \geq \frac{m+1}{n}\sum_{v\in S}\textbf{y}_v - 6c\Delta\ell^{1-\zeta}$.
        }
    \begin{equation}\label{equation: Z_0}
		Z_0 = \frac{m+1}{n} \sum_{v\in S} \textbf{y}_v \pm 6c\Delta\ell^{1-\zeta}.
	\end{equation}
	Recall that $Z_k\leq \sum_{j=0}^m X_j \leq Z_k+3c\Delta\ell^{\frac12+\zeta}$ by~\eqref{equation: sum X_j}. Combining this with~\eqref{equation: Z_0} yields
	\begin{align*}
		\left|\sum_{j=0}^m X_j-\frac{m+1}{n} \sum_{v\in S}\textbf{y}_v\right|
		&\leq |Z_k-Z_0| + \left|Z_0-\frac{m+1}{n} \sum_{v\in S}\textbf{y}_v \right| + 3c\Delta\ell^{\frac12+\zeta} \\
        &\leq |Z_k-Z_0| + 9c\Delta n^{\frac14-\frac{\zeta}{16}}.
	\end{align*}
	Thus,
	\begin{align*}
		\pr\left[\left|\sum_{j=0}^m X_j - \frac{m+1}{n}\sum_{v\in S}\textbf{y}_v\right|
		\geq cn^{\frac14-\frac{1}{17\sqrt{\ln n}}}\right]
		&\leq \pr\left[\left|\sum_{j=0}^m X_j - \frac{m+1}{n} \sum_{v\in S}\textbf{y}_v \right|
		\geq 10c\Delta n^{\frac14-\frac{\zeta}{16}}\right] \\
		&\leq \pr\left[\left|Z_k-Z_0\right|\geq c\Delta n^{\frac14-\frac{\zeta}{16}}\right]
		\leq \frac{1}{n^3},
	\end{align*}
    which proves the statement.
\end{proof}

In particular, by using Lemma~\ref{lemma: neighbours tree}, we obtain that a random tree is very likely to visit a vertex set~$S$ around the expected amount of times. Furthermore, given a vertex $v\in V(G)$, the sum of edge weights to neighbours in~$S$ and the entropy of these edges are very likely to be close to their expected value, respectively.

\begin{cor}\label{cor: azuma corollary}
Under the assumptions of Lemma~\ref{lemma: neighbours tree}, by choosing $\textbf{y}\equiv 1$ and $c=1$, we obtain
    \begin{align*}
        \pr\left[\left|\sum_{j=0}^m \1_{R_j\in S} - \frac{m+1}{n}|S|\right| \geq n^{\frac14-\frac{1}{17\sqrt{\ln n}}}\right] \leq \frac{1}{n^3}
    \end{align*}
for every $S\subseteq V(G)$.
Let $v\in V(G)$.
Choosing $S:=N^+(v)$, $\textbf{y}_w := \textbf{x}_{vw}$ for $w\in V(G)$ and $c=\frac{b}{n}$ yields
    \begin{align*}
        \pr\left[\left|\sum_{j=0}^m\textbf{x}_{vR_j}\1_{R_j\in N^+(v)} - \frac{m+1}{n}\right| \geq n^{-\frac34-\frac{1}{18\sqrt{\ln n}}}\right] \leq \frac{1}{n^3},
    \end{align*}
as well as the analogous result for~$N^-(v)$.
Choosing $S:=N^+(v)$, $\textbf{y}_w := \textbf{x}_{vw}\log\frac{1}{\textbf{x}_{vw}}$ for $w\in V(G)$ and $c=\frac{b}{n}\log n$ yields
    \begin{align*}
	\pr\left[\left|\sum_{j=0}^m\textbf{x}_{vR_j}\log\frac{1}{\textbf{x}_{vR_j}}\1_{R_j\in N^+(v)} - \frac{m+1}{n}h_{\textbf{x}}^+(v)\right| \geq n^{-\frac34-\frac{1}{18\sqrt{\ln n}}}\right] \leq \frac{1}{n^3},
    \end{align*}
as well as the analogous result for~$N^-(v)$ and~$h_{\textbf{x}}^-(v)$.
\end{cor}

The following result combines Lemma~\ref{lemma: self-avoiding} and Corollary~\ref{cor: azuma corollary} to obtain a lower bound on the probability that a random tree is self-avoiding and $(\cS,n^{\frac14-\frac{1}{17\sqrt{\ln n}}},n^{-\frac34-\frac{1}{18\sqrt{\ln n}}})$-expected for a collection~$\cS$ of subsets of~$V(G)$ of size at most~$n^2$.
\begin{cor}\label{cor: well-behaved tree}
	Suppose $\frac1n\ll\gamma\ll\frac1b\ll\eps$. Let~$G$ be an \mbox{$(n,\eps)$-digraph} and let~$\textbf{x}$ be a \mbox{$b$-normal} perfect fractional matching of~$G$.
	Let $\Delta:= e^{\gamma\sqrt{\ln n}}$,
    let $T=(r_0,\ldots,r_m)$ be an oriented tree with $\Delta(T)\leq\Delta$ and $n^\frac14+1\leq |T|\leq3\Delta n^\frac14$, and let $R=(R_0,\ldots,R_m)$ be a random tree in~$G$ according to~$T$.
	Let~$\cS$ be a collection of subsets of~$V(G)$ of size at most~$n^2$.
	Let~$\cE$ be the event that~$R$ is self-avoiding and \mbox{$(\cS,n^{\frac14-\frac{1}{17\sqrt{\ln n}}},n^{-\frac34-\frac{1}{18\sqrt{\ln n}}})$-expected}. Then
	\begin{equation*}
		\pr[\cE] \geq 1 - m^2\frac{b}{n} - n^2\frac{1}{n^3} - 4n\frac{1}{n^3} \geq 1- n^{-\frac13}.
	\end{equation*}
\end{cor}

\section{Counting trees}\label{sec:Counting trees}

In this section we prove our main theorem which states that every $(n,\eps)$-digraph~$G$ contains $|\op{Aut}(T)|^{-1}2^{h(G)-n\log e-o(n)}$ copies of any oriented $n$-vertex tree~$T$ with maximum degree at most~$e^{o(\sqrt{\log n})}$.

Our proof strategy is as follows:
Suppose $\frac1n\ll\frac{1}{\gamma}\ll\eps$. We set $\Delta:=e^{\gamma\sqrt{\ln n}}$ and
partition~$T$ into vertex-disjoint subtrees $T',T''$ such that $n^{1-\alpha}\leq|T''|\leq\Delta n^{1-\alpha}$, where $\alpha:=\frac{1}{7000\sqrt{\ln n}}$.
Let~$t't''$ be the unique edge joining~$T'$ and~$T''$, where $t'\in V(T')$.
Using Theorem~\ref{theorem: absorption}, we fix a set $A\subseteq V(G)$ of size $|T''|-n^{1-7\alpha}$ and a vertex $v\in A$ such that the following holds. For any set $B\subseteq V(G)$ with $A\subseteq B$ and $|B|=|T''|$, the digraph~$G[B]$ contains a copy of~$T''$ in which~$t''$ is mapped to~$v$.
We decompose~$T'$ into subtrees $T_1,\ldots,T_k$ such that $t'\in T_1$, and any tree~$T_i$ with $i\in\{2,\ldots,k\}$ intersects exactly one tree~$T_j$ with $j<i$, and that intersection is exactly one vertex \textcolor{black}{(see Lemma~\ref{lemma: tree decomposition})}.
Further, if $n_0:=n-|A|$ and $n_i:=n-|A|-|T_1\cup\ldots\cup T_i|+1$ for $i\in[k-1]$, we make sure that $n_{i-1}^{\frac14}+1\leq|T_i|\leq 3\Delta n_{i-1}^{\frac14}$ for each $i\in[k]$.

We then iteratively embed~$T'$ \textcolor{black}{into~$G_0:=G-A$}.
Define~$Q_i$ to be a copy of $T_1\cup\ldots\cup T_i$ in~$G_0$. Let $t\in V(T_{i+1})$ be the unique vertex such that $t\in V(T_j)$ for some $j<i$, and let~$u\in V(G_0)$ be its image in~$G_0$. Let $G_i:=G_0[(V(G_0)\setminus V(Q_i))\cup \{u\}]$.
Now we are interested in the number of extensions of our already defined partial embedding,
that is, we seek copies of $T_{i+1}$ in $G_i$ where $t_{i+1}$ is embedded onto $u$.
Among those extensions we are only interested in those which are suitably random-like (a property we call \emph{well-behaved}) so that we can iterate this procedure.

We prove that any \mbox{$(n,\eps)$-digraph}~$G$ contains a large number of well-behaved copies of any tree~$T$ with $n^\frac14+1 \leq |T| \leq3\Delta n^\frac14$ (Corollary~\ref{cor: one tree}). In particular, $G_0$ contains a large number of well-behaved copies of~$T_1$, where~$t'$ is embedded to a neighbour of~$v$.
We then show that if we embed well-behaved copies of $T_1,\ldots,T_i$, then for~$G_i$ the necessary conditions of Corollary~\ref{cor: one tree} are satisfied (Lemma~\ref{lemma: iteration basis}).
Thus, we are able to iterate the calculations in Corollary~\ref{cor: one tree} to show that each~$G_i$ contains a large number of well-behaved copies of $T_{i+1}$.
This yields the existence of $2^{h(G)-n\log e-o(n)}$ well-behaved copies of~$T'$ in $G-A$ (Theorem~\ref{theorem: large tree}).
Recall that~$t'$ is embedded onto a neighbour of~$v$.
Using the absorbing property of~$A$, for any such copy~$Q'$ of~$T'$, we find a copy of~$T''$ in~$G-Q'$ in which~$t''$ is embedded onto~$v$, thus embedding~$T$ completely into~$G$. By taking into account that we count unlabelled trees, this yields the desired result.

We need the following theorem.
\begin{theorem}[{\cite[Theorem~2.1]{KM:22}}]\label{theorem: absorption}
    Suppose $\frac1n\ll\gamma\ll\eps$. Let~$\mu$ be such that $n^{-\frac{1}{7000\sqrt{\ln n}}} \leq\mu\leq \eps^4$. Let~$G$ be an \mbox{$(n,\eps)$-digraph}, let~$T$ be an oriented tree with $|T|=\mu n$ and $\Delta(T)\leq e^{\gamma\sqrt{\ln n}}$, and let $t\in V(T)$.

    Then there is a vertex set~$A\subseteq V(G)$ of size $|T|-\mu^7 n$ and a vertex $v\in A$ such that for any set $B\subseteq V(G)$ with $A\subseteq B$ and $|B|=|T|$, the digraph~$G[B]$ contains a copy of~$T$ in which~$t$ is mapped to~$v$.
\end{theorem}

The next theorem  yields a lower bound on the entropy of a random tree in an $(n,\eps)$-digraph with a $b$-normal perfect fractional matching~\textbf{x} in terms of~$h(\textbf{x})$.

\begin{theorem}\label{theorem: H(R)}
	Suppose $\frac1n\ll\gamma\ll\frac1b\ll\eps$. Let~$\Delta:= e^{\gamma\sqrt{\ln n}}$, let~$G$ be an \mbox{$(n,\eps)$-digraph}, let~$\textbf{x}$ be a \mbox{$b$-normal} perfect fractional matching of~$G$ and let $v_0\in V(G)$.
	Let $Q=(r_0,\ldots,r_m)$ be an oriented tree with $n^\frac14+1 \leq |Q|\leq 3\Delta n^\frac14$ and $\Delta(Q)\leq\Delta$, and let $R=(R_0,\ldots,R_m)$ be a random tree in~$G$ according to~$Q$, where $R_0=v_0$.
	Then
	\begin{equation*}
		H(R) \geq (1-2e^{-\sqrt{\ln n}}) \frac{m}{n} h(\textbf{x}).
	\end{equation*}
\end{theorem}
\begin{proof}
	We prove this statement by applying the chain rule of entropy to write~$H(R)$ as sum of conditional entropies, where~$H(R_i)$ is only conditioned on the path between~$R_0$ and~$R_i$ in the random tree. We then only consider those vertices in the tree with a certain minimum distance from the root and utilise Corollary~\ref{cor: probability} to find a lower bound of~$H(R)$ in terms of~$h(\textbf{x})$.
	
	\textcolor{black}{Let~$r_j$ be an ancestor of~$r_i$ and~$r_{i'}$.}
	By definition, $R_i\mid R_j$ and~$R_{i'}\mid R_j$ are independent, hence, \mbox{$H(R_i\mid R_j,R_{i'}) = H(R_i\mid R_j)$}.%
    \COMMENT{
    We have
    \begin{align*}
        H(R_i\mid R_j,R_{i'}) &= \sum_{v,w}\pr[R_j=v,R_{i'}=w]H(R_i\mid R_j=v,R_{i'}=w) \\
        &= \sum_{v,w} \pr[R_j=v,R_{i'}=w]\sum_x \pr[R_i=x\mid R_j=v,R_{i'}=w]\log\frac{1}{\pr[R_i=x\mid R_j=v,R_{i'}=w]} \\
        &= \sum_{v,w} \pr[R_j=v,R_{i'}=w]\sum_x \pr[R_i=x\mid R_j=v]\log\frac{1}{\pr[R_i=x\mid R_j=v]} \\
        &= \sum_v \pr[R_j=v]\sum_x \pr[R_i=x\mid R_j=v]\log\frac{1}{\pr[R_i=x\mid R_j=v]} \\
        &= \sum_v \pr[R_j=v] H(R_i\mid R_j=v)
        = H(R_i\mid R_j)
    \end{align*}
    }
	By induction we conclude that for every $i\in[m]$, there is a path $r_0,r_{i_1},\ldots,r_{i_k},r_i$ in~$Q$ such that $H(R_i\mid R_0,\ldots,R_{i-1})=H(R_i\mid R_0,R_{i_1},\ldots,R_{i_k})$.
    For any ordered set $(v_1,\ldots,v_s)$ of vertices, let~$\cA_{(v_1,\ldots,v_s)}$ be the event that $R_{i_1}=v_1,\ldots,R_{i_s}=v_s$.
    Suppose $r_{i_k}r_i\in E(T)$.
    Using the Markov property we obtain
    \begin{align*}
        H(R_i\mid \cA_{(v_{i_1},\ldots,v_{i_k})}) &= \sum_{v\in V(G)} \pr[R_i=v \mid \cA_{(v_{i_1},\ldots,v_{i_k})}] \log\frac{1}{\pr[R_i=v\mid \cA_{(v_{i_1},\ldots,v_{i_k})}]} \\
        &= \sum_{v\in V(G)} \pr[R_i=v\mid R_{i_k}=v_{i_k}] \log\frac{1}{\pr[R_i=v\mid R_{i_k}=v_{i_k}]} \\
        &= \sum_{v\in V(G)} \textbf{x}_{v_{i_k}v}\log\frac{1}{\textbf{x}_{v_{i_k}v}}
        = h_{\textbf{x}}^+(v_{i_k}).
    \end{align*}
    Thus, using the law of total probability,
    \begin{align*}
        H(R_i\mid R_0,\ldots,R_{i_k})
        &= \sum_{v_{i_1},\ldots,v_{i_k}\in V(G)}\pr[\cA_{(v_{i_1},\ldots,v_{i_k})}] H(R_i\mid \cA_{(v_{i_1},\ldots,v_{i_k})}) \\
        &= \sum_{v_{i_1},\ldots,v_{i_{k}}\in V(G)} \pr[R_{i_k}=v_{i_k}\mid \cA_{(v_{i_1},\ldots,v_{i_{k-1}})}]\pr[\cA_{(v_{i_1},\ldots,v_{i_{k-1}})}] h_{\textbf{x}}^+(v_{i_k}) \\
        &= \sum_{v_{i_k}\in V(G)} \pr[R_{i_k}=v_{i_k}]h_{\textbf{x}}^+(v_{i_k})
        = \ex h_{\textbf{x}}^+(R_{i_k}).
    \end{align*}
    If $r_ir_{i_k}\in E(T)$, we analogously obtain $ H(R_i\mid R_0,\ldots,R_{i_k}) = \ex h_{\textbf{x}}^-(R_{i_k})$.

    Recall that $(r_0,\ldots,r_m)$ is a breadth-first ordering of the vertices of~$T$. As $\Delta(T)\leq\Delta$, all children of a vertex~$r_j$ are in the set $\{r_{j+1},\ldots,r_{(j+1)\Delta}\}$.
    Thus for any $a\in\{2,\ldots,m\}$ and $i\geq a$, the parent of~$r_i$ must be in the set $\{r_{\frac{a}{2\Delta}},\ldots,r_{i-1}\}$.\COMMENT{That follows easily from contradiction.} In particular, $H(R_i\mid R_0,\ldots,R_{i-1})\geq \min_{j\geq\frac{a}{2\Delta}}\{\ex h_{\textbf{x}}^+(R_j),\ex h_{\textbf{x}}^-(R_j)\}$.
    Thus, using the chain rule of entropy,
	\begin{equation*}
		H(R) = \sum_{i=1}^{m} H(R_i\mid R_0,\ldots,R_{i-1})\geq (m-a)\min_{j\geq \frac{a}{2\Delta}} \{\ex h_{\textbf{x}}^+(R_j),\ex h_{\textbf{x}}^-(R_j)\}.
	\end{equation*}
	Let $a:=2\Delta n^\frac18$ and $j\geq\frac{a}{2\Delta}$. Then the depth of~$r_j$ is at least $\frac12\log_\Delta j \geq \frac{\ln n}{16\ln\Delta} = \frac{\sqrt{\ln n}}{16\gamma}$.\COMMENT{As~$T$ has maximum degree~$\Delta$, there are at most~$\Delta^j$ vertices of depth~$j$. Thus, the number of vertices of depth~$\leq j$ is at most $\Delta^j+\Delta^{j-1}+\ldots+\Delta+1 < \Delta^{j+1}$ (induction on~$j$). For $j\geq\Delta^2$, this implies that~$r_j$ has depth at least $\frac12\log_\Delta j$ (if~$r_j$ has depth $\leq \frac12\log_\Delta j \leq \log_\Delta j - 1$, the number of vertices of depth $\leq \frac12\log_\Delta j$ is less than~$\Delta^{\log_\Delta j} = j$, a contradiction).}
	Hence, according to Corollary~\ref{cor: probability},
	for each $v\in V(G)$, we have $\pr[R_j=v] = \frac1n(1\pm e^{-\frac{\eps}{32b^2\gamma}\sqrt{\ln n}}) = \frac1n(1\pm e^{-\sqrt{\ln n}})$, as $\gamma\ll\frac1b\ll\eps$.
    Thus, for each $*\in\{+,-\}$ and $j\geq\frac{a}{2\Delta}$, we obtain
	\begin{align*}
		\ex h_{\textbf{x}}^*(R_j)
		&= \sum_{v\in V(G)} h_{\textbf{x}}^*(v) \pr[R_j=v]
		= \frac1n(1\pm e^{-\sqrt{\ln n}}) \sum_{v\in V(G)} h_{\textbf{x}}^*(v)
		= \frac1n(1\pm e^{-\sqrt{\ln n}}) h(\textbf{x}).
	\end{align*}
    This implies\COMMENT{As $a=2\Delta n^\frac18$, we have $a\leq m\cdot \Delta n^{-\frac{1}{8}} \leq mn^{-\frac19} = me^{-\frac19\ln n}$. Furthermore, we have $e^{-\sqrt{\ln n}} \geq e^{-\frac19\ln n}$ and thus $(1-e^{-\frac19\ln n})(1-e^{-\sqrt{\ln n}}) \geq 1-2e^{-\sqrt{\ln n}}$. Thus, $(m-a)(1-e^{-\sqrt{\ln n}}) \geq m(1-e^{-\frac19\ln n})(1-e^{-\sqrt{\ln n}}) \geq m(1-2e^{-\sqrt{\ln n}})$.}
    $H(R) \geq (m-a)\frac1n(1-e^{-\sqrt{\ln n}})h(\textbf{x})
		\geq (1-2e^{-\sqrt{\ln n}}) \frac{m}{n} h(\textbf{x})$, as desired.
\end{proof}

We now define what it means for a random tree to be \emph{$(G,G')$-well-behaved}.
Let~$G$ be a digraph on~$n$ vertices and~$G'$ an induced subdigraph of~$G$ on~$n'$ vertices.
We define $\cS_{G,G'}:=\{N_{G,G'}^{*}(v) \colon v\in V(G), *\in\{+,-\}\}$.
We say that a set $M\subseteq V(G')$ with $|M|\geq (n')^\frac14$ is \emph{$(G,G')$-well-behaved} if it is $(\cS_{G,G'},(n')^{\frac14-\frac{1}{17\sqrt{\ln n'}}},(n')^{-\frac34-\frac{1}{18\sqrt{\ln n'}}})$-expected, a notion which is defined in Definition~\ref{defin: (S,a,c)-expected}.
We call a subgraph~$Q$ of~$G'$ \emph{well-behaved} if~$V(Q)$ is well-behaved.
If~$R$ is a random tree in~$G'$, we define~$R$ to be \emph{well-behaved} if~$R$ is self-avoiding and if~$V(R)$ is $(G,G')$-well-behaved.

Now let~$G$ be an $(n,\eps)$-digraph and~$G'$ an induced $(n',\eps)$-subdigraph of~$G$ with $n'\geq n^{\frac23}$, and let~$\bx$ be a $b$-normal perfect fractional matching of~$G'$.
By applying Theorem~\ref{theorem: H(R)} we obtain the following corollary, which states that~$G'$ contains essentially at least $2^{\frac{m}{n'}h(\textbf{x})}$ copies of any labelled $(G,G')$-well-behaved tree on~$m$ vertices.

\begin{cor}\label{cor: one tree}
    Suppose $\frac1n\ll\gamma\ll\frac1b\ll\eps$. Let~$\Delta:= e^{\gamma\sqrt{\ln n}}$, let~$G$ be an \mbox{$(n,\eps)$-digraph}, let~$G'$ be an induced $(n',\eps)$-subdigraph of~$G$ with $n'\geq n^{\frac23}$ and let $v_0\in V(G')$.
    Let $\textbf{x}$ be a \mbox{$b$-normal} perfect fractional matching of~$G'$.
	Let~$Q=(r_0,\ldots,r_m)$ be an oriented tree with $(n')^\frac14+1\leq |Q|\leq 3\Delta (n')^\frac14$ and $\Delta(Q)\leq\Delta$. Let~$\cQ$ be the set of $(G,G')$-well-behaved oriented trees in~$G'$ isomorphic to~$Q$ with root~$v_0$.
    Then
    \begin{equation*}
	   \log|\cQ| \geq (1-e^{-\frac12\sqrt{\ln n'}})\frac{m}{n'}h(\textbf{x}).
    \end{equation*}\end{cor}
\begin{proof}
    Let~$R$ be a random tree in~$G'$ according to~$Q$ that satisfies $R_0=v_0$.
    Note that~$\cQ$ is exactly the set of $(G,G')$-well-behaved realisations of~$R$ in~$G'$.
    Let~$\cE$ denote the event that~$R$ is $(G,G')$-well-behaved and let \textcolor{black}{$\{T_1,\ldots,T_N\}$} be the set of all realisations of~$R$ in~$G'$. As $n'\geq n^{\frac23}$, we have \textcolor{black}{$|\cS_{G,G'}|\leq 2n<(n')^2$} and thus \textcolor{black}{$\pr[\cE]\geq 1-(n')^{-\frac13}$} by Corollary~\ref{cor: well-behaved tree}. As $\textbf{x}_e\geq\frac{1}{bn'}$ for each $e\in E(G')$, we obtain \textcolor{black}{$\pr[R=T_i] \geq \frac{1}{(bn')^m}$} for each $i\in[N]$.
    Lemma~\ref{lemma: entropy} yields
    $\log|\cQ|\geq H(R\mid \cE) \geq H(R) - 2(n')^{-\frac13}m\log(bn') \geq H(R)-(n')^{-\frac{1}{20}}$.
    By applying Theorem~\ref{theorem: H(R)} we obtain
    $\log|\cQ| \geq (1-2e^{-\sqrt{\ln n'}})\frac{m}{n'} h(\textbf{x}) - (n')^{-\frac{1}{20}} \geq (1-e^{-\frac12\sqrt{\ln n'}})\frac{m}{n'}h(\textbf{x})$.%
    \COMMENT{Note that $2e^{-\sqrt{\ln n'}} \frac{m}{n'}h(\textbf{x}) + (n')^{-\frac{1}{20}} \leq 3e^{-\sqrt{\ln n'}} \frac{m}{n'}h(\textbf{x}) \leq e^{-\frac12\sqrt{\ln n'}}\frac{m}{n'}h(\textbf{x})$ and thus the claim follows.}
\end{proof}

The following lemma yields the basis for iterating Corollary~\ref{cor: one tree}. It states that the removal of a well-behaved tree from an digraph~$G$ with large minimum degree and large entropy results in a digraph which still has large minimum degree and large entropy.
\begin{lemma}\label{lemma: iteration basis}
    Suppose $\frac{1}{n},\frac{1}{n'}\ll\gamma\ll\frac{1}{b}\ll\eps',\eps$. Let $\Delta:= e^{\gamma\sqrt{\ln n'}}$.
    Let~$G$ be an $(n,\eps)$-digraph and let~$G'$ be an induced subdigraph of~$G$.
    Suppose $|N^*_{G,G'}(v)|\geq(\frac12+\eps')n'$ for all $v\in V(G)$ and $*\in\{+,-\}$.
    Let~$Q=(q_0,\ldots,q_m)$ be a $(G,G')$-well-behaved oriented tree in~$G'$ with $(n')^\frac14 + 1 \leq |Q| \leq 3\Delta (n')^\frac14$, and $\Delta(Q)\leq\Delta$.
    Let $n'':=n'-m$ and $\eps'':=\eps'-2(n')^{-\frac34-\frac{1}{18\sqrt{\ln n'}}}$.
    Let $u\in V(G)\setminus V(G')$.
    Then for the digraph $G'':=G[(V(G')\setminus V(Q))\cup\{u\}]$ and for each $v\in V(G)$ and $*\in\{+,-\}$ the following holds.
    \begin{enumerate}[label=\normalfont(\roman*)]
        \item $|N_{G,G'}^{*}(v)\cap V(Q)| = \frac{m+1}{n'}|N_{G,G'}^{*}(v)|\pm (n')^{\frac14-\frac{1}{17\sqrt{\ln n'}}}$, in particular,
        \item $\abs{N_{G,G''}^{*}(v)}\geq(\frac12+\eps'')n''$, in particular,
        \item $G''$ is an \mbox{$(n'',\eps'')$-digraph}.
    \end{enumerate}
    If~$\textbf{x}'$ is a \mbox{$b$-normal} perfect fractional matching of~$G'$, we additionally obtain that
    \begin{enumerate}[label=\normalfont(\roman*)]
    \setcounter{enumi}{3}
        \item $G''$ has a $(1+(n')^{-\frac34-\frac{1}{22\sqrt{\ln n'}}})b$-normal perfect fractional matching $\textbf{x}''$ that satisfies $h(\textbf{x}'')\geq \frac{n''}{n'}h(\textbf{x}') -n''\log\frac{n'}{n''} - (n')^{\frac14-\frac{1}{24\sqrt{\ln n'}}}$.
    \end{enumerate}
\end{lemma}
\begin{proof}
    As~$Q$ is a $(G,G')$-well-behaved tree, (i) follows. As $N_{G,G''}^{*}(v) = N_{G,G'}^*(v)\setminus (N_{G,G'}^*(v)\cap V(Q))$, this yields%
        \COMMENT{
            We have $\frac{|S|}{n'}+(n')^{\frac14-\frac{1}{17\sqrt{\ln n'}}} \leq 1+(n')^{\frac14-\frac{1}{17\sqrt{\ln n'}}} \leq (n')^{\frac14-\frac{1}{18\sqrt{\ln n'}}}$.
        }
    \begin{align*}
        |N_{G,G''}^{*}(v)| &= |N_{G,G'}^*(v)| - |N_{G,G'}^*(v)\cap V(Q)| > \frac{n'-(m+1)}{n'}|N_{G,G'}^*(v)| - (n')^{\frac14-\frac{1}{17\sqrt{\ln n'}}} \\
        &\geq \frac{|N_{G,G'}^*(v)|}{n'}n'' - (n')^{\frac14-\frac{1}{18\sqrt{\ln n'}}}
        \geq \left(\frac12+\eps'\right)n'' - (n')^{\frac14-\frac{1}{18\sqrt{\ln n'}}} \\
        &\geq \left(\frac12+\eps''\right)n''
    \end{align*}
    for any $v\in V(G)$.
    To prove~(iv), we apply Lemma~\ref{lemma: new matching} with~$G,G',G'',V(Q)$ playing the role of~$F,G,G',M$, respectively.
    Conditions~(i) and~(ii) in Lemma~\ref{lemma: new matching} follow from~$Q$ being well-behaved, condition~(iii) in Lemma~\ref{lemma: new matching} follows from the assumption that $|N^*_{G,G'}(v)|\geq(\frac12+\eps')n'$ for all $v\in V(G)$ and $*\in\{+,-\}$, while condition~(iv) in Lemma~\ref{lemma: new matching} follows from~(i).
    Thus we are able to apply Lemma~\ref{lemma: new matching} and the claim follows.
\end{proof}

In the next theorem, we iterate Corollary~\ref{cor: one tree} to prove that there are essentially at least $2^{h(G)-n\log e}$ copies of any tree on almost~$n$ vertices.
\begin{theorem}\label{theorem: large tree}
    Suppose $\frac1n\ll\gamma\ll\eps$. 
    Let $\Delta:=e^{\gamma\sqrt{\ln n}}$ and
    let $\alpha:=\frac{1}{7000\sqrt{\ln n}}$.
    Let~$\mu$ be such that $n^{-\alpha}\leq \mu \leq \Delta n^{-\alpha}$.
    Let $(T,t)$ be an oriented labelled rooted tree with $|T|= n-\mu n$ and \textcolor{black}{$\Delta(T)\leq \Delta$}.
    Let~$G$ be an $(n,\eps)$-digraph, $A\subseteq V(G)$ with $|A| = n - |T| - \mu^7 n$, let $u_0\in V(G)\setminus A$.
    Let~$\cQ$ be the set of trees in~$G-A$ isomorphic to~$T$ where~$t$ is mapped to~$u_0$.
    Then $|\cQ| \geq 2^{h(G) - n\log e - \eps n}$.
\end{theorem}
\begin{proof}
Note that $G_0:=G-A$ is a $(|G_0|,\frac\eps2)$-digraph.
Let~$b$ such that $ \gamma \ll \frac{1}{b} \ll \eps$.
Let~$\textbf{x}_0$ be a $\frac{b}{2}$-normal perfect fractional matching of~$G_0$ with $h(\textbf{x}_0)\geq h(G)-\frac{\eps}{2}n$ (note that $|A|\leq \frac{n}{\log^2n}$ and thus the existence of such a matching follows from Lemma~\ref{lemma: matching G-A and G}).

We apply Lemma~\ref{lemma: tree decomposition} with~$n-\mu n$ playing the role of~$n$ to decompose~$(T,t)$ into rooted subtrees $(T_1,t_1),\ldots,(T_k,t_k)$ such that the following holds where $n_0:=n-|A|$ and $n_i:=n-|A|-|T_1\cup\ldots\cup T_i|+1$ for $i\in[k-1]$.
    \begin{enumerate}[label=\normalfont(\roman*)]
        \item $V(T) = \bigcup_{i=1}^k V(T_i)$,
        \item $T_i\subseteq T(t_i)$ for every $i\in[k]$,
        \item the depth of $t_1,\ldots,t_k$ is non-decreasing,
         \item for all $i\in\{2,\ldots,k\}$, there is a unique $j<i$ such that~$T_i$ and~$T_j$ intersect, and in this case we have $V(T_i)\cap V(T_j)=\{t_i\}$, and
        \item $n_{i-1}^\frac14+1 \leq |T_i|\leq 3\Delta n_{i-1}^\frac14$ for every $i\in[k]$.
    \end{enumerate}
In particular, $t=t_1$ and $T_1\cup\ldots\cup T_i$ is a tree for every $i\in[k]$.

Let $n_k:=n-|A|-|T_1\cup\ldots\cup T_k| \geq n^{1-7\alpha}$.
As $|T_i|\geq n_k^{\frac14} \geq n^{\frac14-\frac{7\alpha}{4}}$ for each $i\in[k]$, we have $k\leq \frac{n}{\min_{i\in[k]}|T_i|}\leq n^{\frac34+\frac{7\alpha}{4}}$.
Further, as $n_i \geq n_k$ and thus $n\leq n_i^2$ for any $i\in[k]_0$, the tree~$T$ has maximum degree at most $\Delta = e^{\gamma \sqrt{\ln n}} \leq  e^{\gamma\sqrt{\ln n_i^{2}}} = e^{\sqrt{2}\gamma\sqrt{\ln n_i}}$ for any $i\in[k]_0$.
By replacing~$\gamma$ with~$\sqrt{2}\gamma$, we can assume that~$T$ (and thus each subtree~$T_i$) has maximum degree at most~$e^{\gamma\sqrt{\ln n_i}}$ for each~$i\in[k]_0$.

We iteratively embed~$T$ into~$G_0$. To be more precise, in step~$i$ we embed the tree~$T_i$ using the properties of Lemma~\ref{lemma: iteration basis}.
Let~$T_0$ be the tree with vertex set~$\{t_1\}$.
Suppose that $i\in[k-1]_0$ and~$Q_i$ is a copy of $T_0\cup\ldots\cup T_i$ in~$G_0$. Let $G_i':=G_0-Q_i$.
As there is a tree~$T_j$ with $j<i+1$ and $t_{i+1}\in V(T_j)$ by (iv), the image~$u_{i+1}$ of~$t_{i+1}$ in~$G$ satisfies $u_{i+1}\not\in V(G_i')$.
Thus, to embed~$T_{i+1}-t_{i+1}$ into~$G_i'$ to obtain a copy of $T_0\cup\ldots\cup T_{i+1}$, we temporarily add~$u_{i+1}$ to~$G_i'$ and then embed~$T_{i+1}$.
Define $G_i:=G[V(G_i')\cup\{u_{i+1}\}]$.
We further define $G_k:=G_0-Q$ for a copy~$Q$ of~$T$ in~$G_0$.
Note that $n_i=|V(G_i)|$ for each~$i\in[k]_0$.
We employ Lemma~\ref{lemma: iteration basis} to show that this iterative procedure is indeed possible.

In order to make this iteration possible, we show that the important parameters do not degenerate too quickly.
Let~$\bx_i$ be the perfect fractional matching of~$G_i$ given by Lemma~\ref{lemma: iteration basis}~(iv).
Let~$\eps_i$ be such that~$G_i$ is an $(n_i,\eps_i)$-digraph as given by Lemma~\ref{lemma: iteration basis}~(iii).
Lemma~\ref{lemma: iteration basis}~(iv) implies that 
\begin{equation}\label{equation:h(x_i)}
    h(\textbf{x}_i)\geq \frac{n_i}{n_{i-1}}h(\textbf{x}_{i-1}) -n_i\log\frac{n_{i-1}}{n_i} - n_{i-1}^{\frac14-\frac{1}{24\sqrt{\ln n_{i-1}}}}.
\end{equation}

Next we show that $\eps_i\geq \frac{\eps}{4}$ for all $i\in [k]$.
Note that, as $k\leq n^{\frac34+\frac{7\alpha}{4}}$, we obtain $2kn_i^{-\frac34-\frac{1}{18\sqrt{\ln n_i}}} \leq 2k(n^{1-7\alpha})^{-\frac34-\frac{1}{18\sqrt{\ln n^{1-7\alpha}}}} \leq n^{14\alpha -\frac{1}{20\sqrt{\ln n}}} \leq n^{-\frac{1}{30\sqrt{\ln n}}}$ for each $i\in[k]$.%
\COMMENT{
    We have $2k(n^{1-7\alpha})^{-\frac34-\frac{1}{18\sqrt{\ln n^{1-7\alpha}}}}
    \leq 2n^{\frac34+\frac{7\alpha}{4}}n^{-\frac34-\frac{1}{18\sqrt{\ln n^{1-7\alpha}}}+\frac{21\alpha}{4}+\frac{7\alpha}{18\sqrt{\ln n^{1-7\alpha}}}}
    \leq 2n^{14\alpha - \frac{1}{19\sqrt{\ln n}}}
    \leq n^{14\alpha - \frac{1}{20\sqrt{\ln n}}}$
}
Observe that~$G_0$ is an $(n_0,\frac{\eps}{2})$-digraph and thus Lemma~\ref{lemma: iteration basis}~(iii) yields%
\COMMENT{This calculation is the reason why we have to choose~$\alpha\leq\frac{\zeta}{\sqrt{\ln n}}$ with a small constant~$\zeta$.}
    \begin{align*}
        \eps_i
        \geq \frac\eps2-2kn_i^{-\frac34-\frac{1}{18\sqrt{\ln n_i}}}
        \geq \frac{\eps}{2} - n^{-\frac{1}{30\sqrt{\ln n}}}
        > \frac\eps4
    \end{align*}
    for each $i\in[k]$.
    By replacing~$\eps$ with $\frac{\eps}{4}$, we can assume that~$G$ is an~$(n,\eps)$-digraph and that each~$G_i$ is an $(n_i,\eps)$-digraph.
    
    Next we show that~$\bx_i$ is suitably normal for each~$i\in[k]_0$.
    To this end, we first calculate an upper bound for the normality of~$\bx_i$ in~$G_i$ for the smallest~$i$ where $n_i\leq \frac{n}{2}$.
    As for all $j\leq i$, the tree~$T_j$ has at least~$(\frac{n}{3})^{\frac14}$ vertices, we conclude that $i\leq n(\frac{n}{3})^{-\frac14}\leq 2n^{\frac34}$.
    Hence, by Lemma~\ref{lemma: iteration basis} (iv), $\bx_i$ is $b_i$-normal with
    \begin{align*}
        b_i
        &\leq  
        \left(1+\left(\frac{n}{3}\right)^{-\frac34-\frac{1}{22\sqrt{\ln \frac{n}{3}}}}\right)^{2n^{\frac34}} \cdot \frac{b}{2}
        \leq \exp(3n^{-\frac34-\frac{1}{22\sqrt{\ln n}}} \cdot 2n^{\frac34})\cdot \frac{b}{2}
        = \exp(6n^{-\frac{1}{22\sqrt{\ln n}}})\cdot \frac{b}{2} \\
        &\leq \exp(n^{-\frac{1}{23\sqrt{\ln n}}})\cdot \frac{b}{2}.
    \end{align*}
    This means that the normality grows by at most a multiplicative factor of $\exp(n^{-\frac{1}{23\sqrt{\ln n}}})$
    for each time the number of vertices of our considered graph shrinks by a factor of~$2$.
    We can iterate this calculation until we reach~$G_k$. For this, we need at most~$\log n$ iterations.
    As $\exp(n_k^{-\frac{1}{23\sqrt{\ln n_k}}}\log n)\leq 2$, we conclude that $\bx_i$ is $b$-normal for each $i\in [k]_0$.

    Recall that~$\cQ$ denotes the set of all trees in~$G_0$ isomorphic to~$T$ where~$t$ is mapped to~$u_0$.
    We now apply Corollary~\ref{cor: one tree} to each~$G_i$ and use \eqref{equation:h(x_i)}. This yields a lower bound for~$|\cQ|$ in terms of~$h(\textbf{x}_0)$. A suitable estimate for the remaining terms yields the desired lower bound.
    
    Suppose $T_0,\ldots,T_{i}$ have already been embedded in~$G_0$. Recall that $t_{i+1}\in V(T_0\cup\ldots\cup T_{i})$ and~$u_{i+1}$ denotes the image of~$t_{i+1}$ in~$G_0$.
    Denote by~$\cQ_{i+1}$ the set of $(G,G_i)$-well-behaved trees in~$G_i$ isomorphic to~$T_{i+1}$ where~$t_{i+1}$ is mapped to~$u_{i+1}$.
    Let~$Q^*_{i+1}:=\min|\cQ_{i+1}|$, where the minimum is taken over all possible embeddings of~$T_0\cup\ldots\cup T_i$ in~$G_0$ where~$t$ is mapped to~$u_0$.
    Note that $|\cQ|\geq\prod_{i=1}^kQ_i^*$.
    Let $m_i:=|T_{i}|-1$.
    Applying Corollary~\ref{cor: one tree} to~$G_i$ for each~$i\in[k-1]_0$ yields
    \begin{align}\label{equation: log Q}
        \log |\cQ| &\geq \sum_{i=1}^k \log Q_i^*
        \geq \sum_{i=0}^{k-1} (1-e^{-\frac12\sqrt{\ln n_i}}) \frac{m_{i+1}}{n_i} h(\textbf{x}_i)
        \geq (1-e^{-\frac14\sqrt{\ln n}}) \sum_{i=0}^{k-1} \frac{m_{i+1}}{n_i} h(\textbf{x}_i),
    \end{align}
    where we used $n_i\geq n^{1-7\alpha} \geq n^\frac14$ and thus $e^{\sqrt{\ln n_i}} \geq e^{\frac12\sqrt{\ln n}}$ in the last inequality.
    Next we bound~$h(\textbf{x}_i)$ from below.
    Using \eqref{equation:h(x_i)} and $\log\frac{n_{i-1}}{n_i}+\log\frac{n_{i-2}}{n_{i-1}} = \log\frac{n_{i-2}}{n_i}$, we obtain
    \begin{align*}
        h(\textbf{x}_i) &\geq \frac{n_i}{n_{i-1}}h(\textbf{x}_{i-1}) - n_i\log\frac{n_{i-1}}{n_i} - n_{i-1}^{\frac14-\frac{1}{24\sqrt{\ln n_{i-1}}}} \\
        &\geq \frac{n_i}{n_{i-1}}\left(\frac{n_{i-1}}{n_{i-2}}h(\textbf{x}_{i-2}) - n_{i-1}\log\frac{n_{i-2}}{n_{i-1}} - n_{i-2}^{\frac14-\frac{1}{24\sqrt{\ln n_{i-2}}}}\right) - n_i\log\frac{n_{i-1}}{n_i} - n_{i-1}^{\frac14-\frac{1}{24\sqrt{\ln n_{i-1}}}} \\
        &\geq \frac{n_i}{n_{i-2}}h(\textbf{x}_{i-2}) - n_i\log\frac{n_{i-2}}{n_i} - 2n^{\frac14-\frac{1}{24\sqrt{\ln n}}}.
    \end{align*}
    Iterating this calculation yields
    \begin{align}\label{equation: h(x_i) the second}
        h(\textbf{x}_i)
        &\geq \frac{n_i}{n_0}h(\textbf{x}_0) - n_i\log\frac{n_0}{n_i} - in^{\frac14-\frac{1}{24\sqrt{\ln n}}}
        \geq \frac{n_i}{n_0}h(\textbf{x}_0) - n_i\log\frac{n_0}{n_i} - kn^{\frac14-\frac{1}{24\sqrt{\ln n}}}
    \end{align}
    for each $i\in[k]$.
    Recall that $n_i\geq n^{1-7\alpha}$ and $m_{i+1}\leq 3\Delta n_i^{\frac14}$. 
    Thus $\frac{m_{i+1}}{n_i}\leq 3\Delta n_i^{-\frac34} \leq 3\Delta n^{-\frac34+\frac{21\alpha}{4}} \leq 3n^{-\frac34+\frac{21\alpha}{4}+\frac{\gamma}{\sqrt{\ln n_i}}}$ and hence $\frac{m_{i+1}}{n_i}kn^{\frac14-\frac{1}{24\sqrt{\ln n}}} \leq 3n^{-\frac34+\frac{21\alpha}{4}+\frac{\gamma}{\sqrt{\ln n_i}}}n^{\frac34+\frac{7\alpha}{4}}n^{\frac14-\frac{1}{24\sqrt{\ln n}}} 
    = 3n^{\frac14+7\alpha+\frac{\gamma}{\sqrt{\ln n_i}}-\frac{1}{24\sqrt{\ln n}}}
    \leq n^{\frac14-\frac{1}{100\sqrt{\ln n}}}$.
    Combining this with~\eqref{equation: h(x_i) the second} yields
    \begin{align*}
        \frac{m_{i+1}}{n_i} h(\textbf{x}_i)
        \geq \frac{m_{i+1}}{n_0}h(\textbf{x}_0) - m_{i+1}\log\frac{n_0}{n_i} - n^{\frac14-\frac{1}{100\sqrt{\ln n}}}.
    \end{align*}
    Hence \eqref{equation: log Q} implies
    \begin{equation}\label{equation: log Q the second}
        \log|\cQ|
        \geq (1-e^{-\frac14\sqrt{\ln n}})\sum_{i=0}^{k-1} \frac{m_{i+1}}{n_0}h(\textbf{x}_0) - \sum_{i=0}^{k-1} m_{i+1}\log\frac{n_0}{n_i}-kn^{\frac14-\frac{1}{100\sqrt{\ln n}}}.
    \end{equation}
    We bound the first term from below.
    Note that $\frac{1}{n_0}(m_1+\ldots+m_k)=\frac{|T|-1}{n_0}\geq\frac{n-\Delta n^{1-\alpha}-1}{n} \geq 1-e^{-\frac{1}{10^4}\sqrt{\ln n}}$%
    \COMMENT{$\Delta n^{-\alpha} \leq n^{\frac{\gamma}{\sqrt{\ln n}}-\frac{1}{7000\sqrt{\ln n}}} = e^{\gamma\sqrt{\ln n}-\frac{1}{7000}\sqrt{\ln n}} \leq e^{-\frac{1}{10^4}\sqrt{\ln n}}$ if we choose $\gamma\leq \frac{1}{7000}-\frac{1}{10^4}$}
    and thus
    \begin{equation}\label{equation: log Q first term}
        (1-e^{-\frac14\sqrt{\ln n}})\sum_{i=0}^{k-1} \frac{m_{i+1}}{n_0}h(\textbf{x}_0) \geq (1-e^{-\frac14\sqrt{\ln n}})(1-e^{-\frac{1}{10^4}\sqrt{\ln n}})h(\textbf{x}_0) \geq (1-e^{-\frac{1}{10^5}\sqrt{\ln n}}) h(\textbf{x}_0).
    \end{equation}
    We bound the second term of~\eqref{equation: log Q the second} from below. As $n-n_k= m_1+\ldots+m_k+1$, we obtain $\sum_{i=0}^{k-1}m_{i+1}\log n_0 \leq \log n\cdot\sum_{i=0}^{k-1} m_{i+1} \leq n\log n - n_k\log n$.
    Note that for each $i\in[k-1]_0$,
    \begin{align*}
        m_{i+1}\log n_i &\geq \log n_i + \log (n_i-1) + \ldots + \log (n_i-m_{i+1}+1) \\
        &= \log (n_i(n_i-1)\cdot\ldots\cdot(n_{i+1}+1)).
    \end{align*}
    Thus,
    \begin{align*}
        \sum_{i=0}^{k-1}m_{i+1}\log n_i
        &\geq \log(n_0(n_0-1)\cdot\ldots\cdot (n_k+1))
        \geq \log\frac{n!}{n^{n-n_0}n_k!} \\
        &= \log n! - \log n_k! - (n-n_0)\log n.
    \end{align*}
   As $n-n_0=|A|\leq \Delta n^{1-\alpha}$ and $\log n! = n\log n - n\log e \pm \log n$ by Stirling's approximation, we obtain%
   \COMMENT{We have $2\log n + \Delta n^{1-\alpha} \log n \leq \Delta^2 n^{1-\alpha} \leq n^{1-\alpha+\frac{2\gamma}{\sqrt{\ln n}}}\leq n^{1-\frac{\alpha}{2}}$}
    \begin{align*}
        \sum_{i=0}^{k-1}m_{i+1}\log n_i &\geq n\log n - n\log e - n_k\log n_k + n_k\log e - 2\log n - \Delta n^{1-\alpha}\log n \\
        &\geq n\log n - n\log e - n_k\log n_k - n^{1-\frac{\alpha}{2}}.
    \end{align*}
    This yields
    \begin{align}\label{equation: log Q second term}
        \sum_{i=0}^{k-1} m_{i+1}\log\frac{n_0}{n_i}
        &\leq - n_k\log n + n\log e + n_k\log n_k + n^{1-\frac{\alpha}{2}} \leq n\log e + n^{1-\frac{\alpha}{2}}.
    \end{align}
    Combining~\eqref{equation: log Q the second} with~\eqref{equation: log Q first term} and~\eqref{equation: log Q second term}, we obtain
    \begin{align*}
        \log |\cQ|
        &\geq (1-e^{-\frac{1}{10^5}\sqrt{\ln n}})h(\textbf{x}_0) - n\log e - n^{1-\frac{\alpha}{2}} - kn^{\frac14-\frac{1}{100\sqrt{\ln n}}} \\
        &\geq h(\textbf{x}_0) - n\log e - n^{1-\frac{1}{10^6\sqrt{\ln n}}}.
    \end{align*}
    As we chose~$h(\textbf{x}_0)$ such that $h(\textbf{x}_0)\geq h(G) - \frac{\eps}{2} n$, this yields the desired result.
\end{proof}

We are now able to prove our main theorem, which follows from Theorem~\ref{theorem: large tree} and Theorem~\ref{theorem: absorption}. 

\begin{proof}[Proof of Theorem~\ref{thm:main_directed}]
Suppose $\frac{1}{n}\ll\gamma\ll\eps$. Let~$G$ be an $(n,\eps)$-digraph, let $\Delta:=e^{\gamma\sqrt{\ln n}}$ and~$T$ be an oriented tree on~$n$ vertices with $\Delta(T)\leq \Delta$.
Label the vertices of~$T$ arbitrarily and choose a vertex~$r\in V(T)$ as root of~$T$.
Let $\alpha:=\frac{1}{7000\sqrt{\ln n}}$.
Choose~$t''$ at maximal distance from~$r$ subject to $|T(t'')|\geq n^{1-\alpha}$. Then $n^{1-\alpha}\leq|T(t'')|\leq\Delta n^{1-\alpha}$. Note that~$t''$ has exactly one neighbour~$t'$ in $V(T)\setminus V(T(t''))$. Define $T':=T-T(t'')$ and $T'':=T(t'')$.

Let $\mu:=\frac{|T''|}{n}$. Note that $\mu\leq e^{\gamma\sqrt{\ln n}-\frac{1}{7000}\sqrt{\ln n}} \leq e^{-\frac{1}{10^4}\sqrt{\ln n}}<\eps^4$ as $\gamma\ll\eps$. We apply Theorem~\ref{theorem: absorption} with~$T''$ playing the role of~$T$ and~$t''$ playing the role of~$t$ to fix a vertex set~$A\subseteq V(G)$ of size $|T''|-\mu^7n$ and a vertex $v\in A$ such that for any set $B\subseteq V(G)$ with $A\subseteq B$ and $|B|=|T''|$, the digraph~$G[B]$ contains a copy of~$T''$ in which~$t''$ is mapped to~$v$. Further, let $w\in N^-(v)$ if $t't''\in E(T)$, otherwise let $w\in N^+(v)$.

As~$T'$ satisfies the conditions of Theorem~\ref{theorem: large tree}, we find at least $2^{h(G)-n\log e-\eps n}$ copies of~$T'$ in~$G-A$ such that~$t'$ is mapped to~$w$. Let~$Q'$ be a tree in~$G-A$ isomorphic to~$T'$.
As $|G-Q'|=|T''|$, using the above mentioned property of~$A$, we find a copy of~$T''$ in $G-Q'$ in which~$t''$ is mapped to~$v$. As~$t'$ is mapped to~$w$, we are able to connect~$T'$ and~$T''$ in~$G$, embedding~$T$ completely into~$G$.

In order to count the number of subgraphs in~$G$ isomorphic to~$T$, 
one has to divide the number of labelled copies of $T$ in $G$ by the number of automorphisms of $T$, which yields the desired result.
\end{proof}

\begin{proof}[Proof of Theorem~\ref{thm:main_undirected}]
Suppose $\frac{1}{n}\ll\gamma\ll\eps$. Let~$G$ be an $(n,\eps)$-graph. 
Let~$D$ be the digraph that arises from~$G$ by duplicating every edge of~$G$ and orienting them in opposite directions.
To be more precise, we set $V(D):=V(G)$ and $E(D):=\{(v,w)\mid vw\in E(G)\}$. 
If~$\textbf{x}$ is any perfect fractional matching of~$G$, we obtain a perfect fractional matching~$\textbf{y}$ of~$D$ by setting $\textbf{y}_{(v,w)}:=\textbf{x}_{vw}$. Then $h(\textbf{y})= 2h(\textbf{x})$, and, in particular, $h(D)\geq 2h(G)$.
Let~$T$ be a tree on~$n$ vertices with $\Delta(T)\leq e^{\gamma \sqrt{\ln n}}$. We orient the edges of~$T$ arbitrarily. Call this oriented tree $T_D$. By Theorem~\ref{thm:main_directed}, there are at least $|\op{Aut}(T)|^{-1}2^{h(D)-n\log e-\eps n}\geq |\op{Aut}(T)|^{-1}2^{2h(G)-n\log e-\eps n}$ copies of~$T_D$ in~$D$. Note that any copy~$Q_D$ of~$T_D$ in~$D$ yields a copy~$Q$ of~$T$ in~$G$ by ignoring the orientation of the edges.
This yields the desired result.
\end{proof}

\section{Conclusion}
We have shown that any digraph~$G$ on~$n$ vertices that satisfies~$\delta^0(G)\geq(\frac12+o(1))n$ contains at least $|\op{Aut}(T)|^{-1}(\frac12-o(1))^nn!$ copies of any oriented tree~$T$ on~$n$ vertices with $\Delta(T)\leq e^{o(\sqrt{\log n})}$.
This follows from the following stronger result. 
If $h(G)$ denotes the entropy of~$G$, that is, the maximum entropy of all perfect fractional matchings of~$G$, then~$G$ contains at least~$2^{h(G)-n\log e-\eps n}$ (labelled) copies of~$T$.
The analogous results hold for undirected graphs and trees.

It would be interesting to find an upper bound on the number of copies of spanning trees in (di)graphs. 
We believe that the lower bound in Theorem~\ref{thm:main_directed} is also an upper bound and the inequality is in fact an equality (up to the error term).
Unfortunately, there seems to be no easy way to obtain an upper bound with the methods as it is obtained for Hamilton cycles.

Note that we require~$T$ to satisfy $\Delta(T)\leq e^{o(\sqrt{\log n})}$, whereas~$G$ contains any tree on~$n$ vertices with $\Delta(T)\leq o(\frac{n}{\log n})$.
Extending our results to trees with maximum degree $o(\frac{n}{\log n})$ is another interesting open problem.

\bibliographystyle{amsplain}
\providecommand{\bysame}{\leavevmode\hbox to3em{\hrulefill}\thinspace}
\providecommand{\MR}{\relax\ifhmode\unskip\space\fi MR }
\providecommand{\MRhref}[2]{%
	\href{http://www.ams.org/mathscinet-getitem?mr=#1}{#2}
}
\providecommand{\href}[2]{#2}

\end{document}